\documentclass[11pt]{article}
\usepackage{setspace}

\usepackage{graphicx}
\usepackage{hangcaption}
\usepackage{amssymb}
\usepackage{amsmath}
\usepackage{amsthm}
\usepackage{algorithmic}
\usepackage{algorithm}
\usepackage{subfigure}

\author
{ Megan Owen
\thanks{ {\tt maowen@berkeley.edu}.
University of California Berkeley, Berkeley, CA, 95720.  This work was supported in part by NSF grant DMS-0555268 at Cornell University.  A 2-page extended abstract of a preliminary version of Section 4 was published in the online proceedings of the 17th Fall Workshop on Computational and Combinatorial Geometry (FWCG 2007).}
}

\title{Computing Geodesic Distances in Tree Space}

\newenvironment{my_enumerate}{
\begin{enumerate}
\setlength{\itemsep}{1pt}
\setlength{\parskip}{0pt}
\setlength{\parsep}{0pt}}{\end{enumerate}
}

\theoremstyle{plain}
\newtheorem{thm}{Theorem}[section]
\newtheorem{cor}[thm]{Corollary}
\newtheorem{lem}[thm]{Lemma}
\newtheorem{prop}[thm]{Proposition}

\theoremstyle{definition}
\newtheorem{defn}[thm]{Definition}
\newtheorem{prob}{Problem}

\newtheorem{rem}[thm]{Remark}

\newcommand{\R}{\mathbb{R}}
\newcommand{\T}{\mathcal{T}}
\newcommand{\Or}{\mathcal{O}}

\newcommand{\bs}{\backslash}

\newcommand{\linalg}{\textsc{PathSpaceGeo}}
\newcommand{\DPalg}{\textsc{GeodeMaps-Dynamic}}
\newcommand{\DCalg}{\textsc{GeodeMaps-Divide}}

\providecommand{\abs}[1]{\lvert#1\rvert}
\providecommand{\norm}[1]{\lVert#1\rVert}

\graphicspath{{./}{Figs/}}

\date{}
\begin{document}
\maketitle

\begin{abstract}
We present two algorithms for computing the geodesic distance between phylogenetic trees in tree space, as introduced by Billera, Holmes, and Vogtmann (2001).  We show that the possible combinatorial types of shortest paths between two trees can be compactly represented by a partially ordered set.  We calculate the shortest distance along each candidate path by converting the problem into one of finding the shortest path through a certain region of Euclidean space.  In particular, we show there is a linear time algorithm for finding the shortest path between a point in the all positive orthant and a point in the all negative orthant of $\R^k$ contained in the subspace of $\R^k$ consisting of all orthants with the first $i$ coordinates non-positive and the remaining coordinates non-negative for $0 \leq i \leq k$. 

\end{abstract}

\pagenumbering{arabic}
\section{Introduction}
Phylogenetic trees, or phylogenies, are used throughout biology to understand the evolutionary history of organisms ranging from primates to the HIV virus.  Outside of biology, they are used in studying the evolution of languages and culture, for example.  Often, reconstruction methods give multiple plausible phylogenetic trees on the same set of taxa, which we wish to compare using a quantitative distance measure.  A more general open question is how best to analyze sets of trees in a statistically rigourous manner, for example, by providing confidence intervals for the generated trees.  The tree space of Billera, Holmes, and Vogtmann \cite{BHV01} and its corresponding geodesic distance measure were developed to provide a framework for addressing these issues (\cite{Holmes03} and \cite{Holmes05}).  In this paper, we give several combinatorial and metric properties of this space in the process of developing two practical algorithms for computing this distance. 


There are many different algorithms to construct phylogenetic trees from biological data  (\cite{DEKM98} and its references), but their accuracy can be affected by such factors as the underlying tree shape  \cite{HP89} or the rate of mutation in the DNA sequences used \cite{KF94}.  To compare these methods through simulation, or to find the likelihood that a certain tree is generated from the data, researchers need to be able to compute a biologically meaningful distance between trees \cite{KF94}.  Several different distances between phylogenetic trees have been proposed (e.g. \cite{DHJLT99}, \cite{EstabrookMcMorrisMeacham85}, \cite{H90}, \cite{R71}, \cite{RF81}).  With the exception of the weighted Robinson-Foulds distance \cite{RF79}, none of these distances incorporate tree edge lengths.
 
In response to the need for a distance measure between phylogenetic trees that naturally incorporates both the tree topology and the lengths of the edges, Billera et al.  \cite{BHV01} introduced the \emph{geodesic distance}.   This distance measure is derived from the tree space, $\T_n$, which contains all phylogenetic trees with $n$ leaves.  The tree space is formed from a set of Euclidean regions, called orthants, one for each topologically different tree.  Two regions are connected if their corresponding trees are considered to be neighbours.  Each phylogenetic tree with $n$ leaves is represented as a point within this space.  There is a unique shortest path, called the \emph{geodesic}, between each pair of trees.  The length of this path is our distance metric.  

The most closely related work is by Staple \cite{S04} and Kupczok et al. \cite{KHK08}, who developed algorithms to compute the geodesic distance based on the notes of Vogtmann \cite{V07}.  Both of these algorithms are exponential in the number of different edges in the two trees.  Although Kupczok et al. developed their algorithm \textsc{GeoMeTree} independently, it can be considered a direct improvement to the algorithm of Staple.  We show in Section 5 that our algorithm performs significantly better than \textsc{GeoMeTree}, although it is still exponential.  A polynomial time, $\sqrt{2}$-approximation of the geodesic distance was given by Amenta et al. \cite{AGPS07}.  Since the submission of this paper, a polynomial time algorithm has been developed to compute the geodesic distance \cite{OwenProvan11}.

Our primary contribution is the three main combinatorial and geometric ideas behind the two algorithms we give for computing the geodesic distance.  First, the candidate shortest paths between trees can be represented as an easily constructible partially ordered set, giving information about the combinatorics of the tree space.  Second, we can find the length of each candidate shortest path by translating the problem into one of finding the shortest path through a region of a lower dimensional Euclidean space.  The solution to this new problem is a linear algorithm for a special case of the Euclidean shortest-path problem with obstacles.  Since the general problem is NP-hard for dimensions greater than 2, this result is also of interest to computational geometers.  Finally, we show that the combinatorics of the geodesic depend on the combinatorics of the geodesic between two simpler trees.  This observation makes it possible to use either a dynamic programming or a divide and conquer approach to significantly reduce the search space.  The two resulting algorithms are computationally practical on some biological data sets of interest.

The remainder of this paper is organized as follows.  In Section 2, we describe the tree space and the geodesic distance.  The problem of finding the geodesic distance has both a combinatorial component, which is investigated in Section 3, and a geometric component, which is covered in Section 4.  More specifically, we introduce a combinatorial framework in Section 3, which represents the candidate shortest paths between trees by an easily constructible partially ordered set (Theorem~\ref{th:one_to_one}).  In Section 4, we translate the problem of calculating the length of a candidate shortest path into a problem in Euclidean space (Theorem~\ref{th:homeo_to_Rm}), and then show that this Euclidean problem can be solved in linear time (Theorem~\ref{th:get_dist} and Theorem~\ref{th:PATHGEOcomplexity}).  Section 5 combines the ideas of Sections 3 and 4 to show that the path taken by a geodesic is related to the geodesic path between two simpler trees (Theorem~\ref{th:can_use_dp}).  This theorem is exploited via dynamic programming and divide and conquer techniques to give two algorithms.
%
%
%
\section{Tree Space and Geodesic Distance}
This section describes the space of phylogenetic trees, $\T_n$, and the geodesic distance.  For further details, see \cite{BHV01}.  A \emph{phylogenetic tree}, or just \emph{tree}, $T = (X, \Sigma)$ is a rooted tree, whose leaves are in bijection with a set of labels $X$ representing different organisms, and whose interior edges are represented by the set $\Sigma$ of non-trivial \emph{splits}.  For this paper, let $X = \{1, ..., n\}$.  The root is labelled with $0$ and sometimes treated like a leaf.  We consider both bifurcating (or binary) trees, in which each interior vertex has degree 3, and multifurcating (or degenerate) trees, in which at least one interior vertex has degree $> 3$. 

A \emph{split} $A | B$ is a partition of $X \cup \{0\}$ into two non-empty sets $A$ and $B$.  A split is in $T$ if it corresponds to some edge $e$ in $T$, such that deleting edge $e$ from $T$ divides $T$ into two subtrees, with one subtree containing exactly the leaves in $A$ and the other subtree containing exactly the leaves in $B$. 
For example, in Figure~\ref{fig:what_is_an_edge}, the split corresponding to the edge $e_3$ partitions the leaves into the sets $\{2,3\}$ and $\{0,1,4,5\}$.  We will refer to a split corresponding to an edge ending in a leaf as a \emph{trivial split}, and to all other splits as simply splits.  A \emph{split of type $n$} is a partition of the set $\{0, 1, ..., n\}$ into two blocks, each containing at least two elements.  If $A \subseteq \Sigma$ is a set of splits in $T$, then let $T /A$ be the tree $T$ with the edges that correspond to $A$ contracted.
%
\begin{figure}[ht]
\centering
\includegraphics[scale = 0.4]{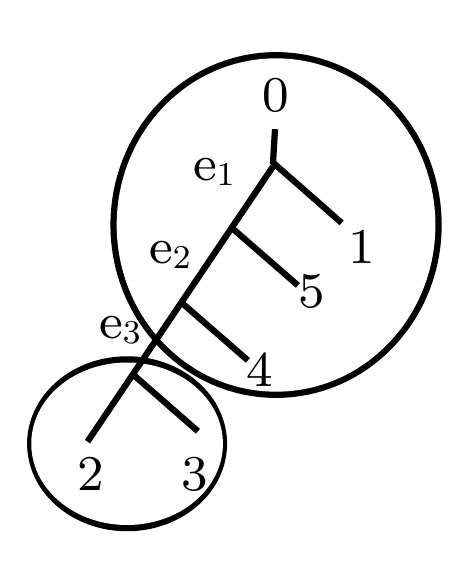}
\caption{The split corresponding to the edge $e_3$.}
\label{fig:what_is_an_edge}
\end{figure}
Two splits $e = X | X'$ and $e' = Y | Y'$ are \emph{compatible} if one of $X \cap Y$, $X \cap Y'$, $X' \cap Y$ or $X' \cap Y'$ is empty.  Equivalently, two splits are compatible if their corresponding edges can exist in the same phylogenetic tree.  For example, in Figure~\ref{fig:what_is_an_edge}, the split $e_3 = \{2, 3\} | \{0, 1, 4, 5\}$ is compatible with the split $e_2  = \{2,3,4\} | \{0,1,5\}$, because $\{2,3\} \cap \{0,1,5\} = \varnothing$.  However, $e_3$ is incompatible with $f = \{1,2\} | \{0,3,4,5\}$.  Two sets of mutually compatible splits of type $n$, $A$ and $B$, are \emph{compatible} if $A \cup B$ is a set of mutually compatible splits. 

For a tree $T = (X, \Sigma)$, each edge, and hence split, $e \in \Sigma$ is associated with a non-negative length $|e|_T$.  For example, this length often represents the expected number of mutations per DNA character site.  
Two splits are considered the same if they have identical partitions, regardless of their associated lengths.  For any set of compatible splits $A \subseteq \Sigma$, let $\norm{A} = \sqrt{\sum_{e \in A} |e|_T^2}$. 
\subsection{Tree Space} \label{sec:tree_space}
We now describe the space of phylogenetic trees, $\T_n$, as constructed by Billera et al. \cite{BHV01}.  It is homeomorphic, but not isometric, to the tropical Grassmannian \cite{SpeyerSturmfels04} and the Bergman fan of the graphic matroid of the complete graph \cite{ArdilaKlivans04}.  This space contains all bifurcating and multifurcating phylogenetic trees with $n$ leaves.  In this space, each tree topology with $n$ leaves is associated with a Euclidean region, called an \emph{orthant}.  The points in the orthant represent trees with the same topology, but different edge lengths.  These orthants are attached, or glued together, to form the tree space.  

We do not use the lengths of the edges ending in leaves in the definition of tree space, but can easily include them by considering geodesics through $\T_n \times \R^n_+$, as noted in Billera et al. \cite{BHV01}.

Any set of $n-2$ compatible splits corresponds to a unique rooted phylogenetic tree topology \cite[Theorem 3.1.4]{SS03}.  For any such split set $\Sigma$ corresponding to tree $T$, associate each split with a vector such that the $n-2$ vectors are mutually orthogonal.  The cone formed by these vectors is the orthant associated with the topology of $T$.  Recall that the $k$-dimensional \emph{(nonnegative) orthant} is the non-negative part of $\R^k$, denoted $\R^k_+$.  A point $(x_1, ..., x_{n-2})$ in $\R^{n-2}_+$ represents the tree in which the edge associated with the $i$-axis has length $x_i$, for all $1 \leq i \leq n-2$, as illustrated in Figure~\ref{fig:orthants}.  If $x_i = 0$, then the tree is on a face of the orthant, and we say that it does not contain the edge associated with the $i$-axis.  Furthermore, two orthants can share the same boundary face, and thus are attached.  For example, in Figure~\ref{fig:orthants}, the trees $T_1$ and $T_1'$ are represented as two distinct points in the same orthant, because they have the same topology, but different edge lengths.  The tree $T_0$ has only one edge, $e_1$, and thus is a point on the $e_1$ axis.  

Notice that although Figure~\ref{fig:orthants} is drawn in the plane, it actually sits in $\R^3$, with each of the axes or splits corresponding to a different dimension.  In general, $\T_n$ sits in $\R^{\mathcal{N}}$, where $\mathcal{N} = 2^n - n - 2$ is the number of possible splits of type $n$.  However, as no point in $\T_n$ has a negative coordinate in $\R^{\mathcal{N}}$, we may draw the positive and negative parts of an axis as corresponding to different splits.

%
\begin{figure}[ht]
\centering
\subfigure[Two orthants in $\T_4$.]{
\includegraphics[scale = 0.3]{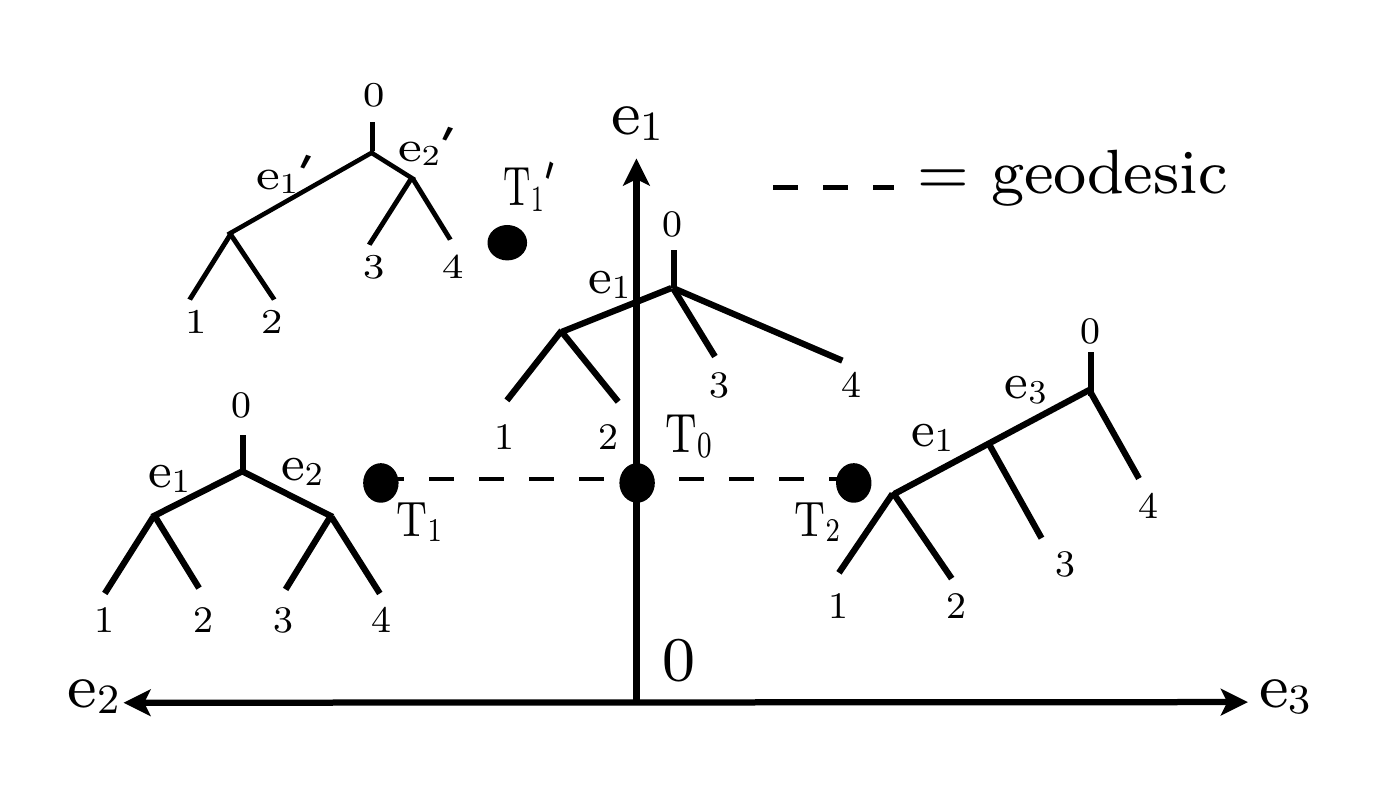}
\label{fig:orthants}
}
\subfigure[Both edge length and tree topology determine the geodesic.]{
\includegraphics[scale = 0.3]{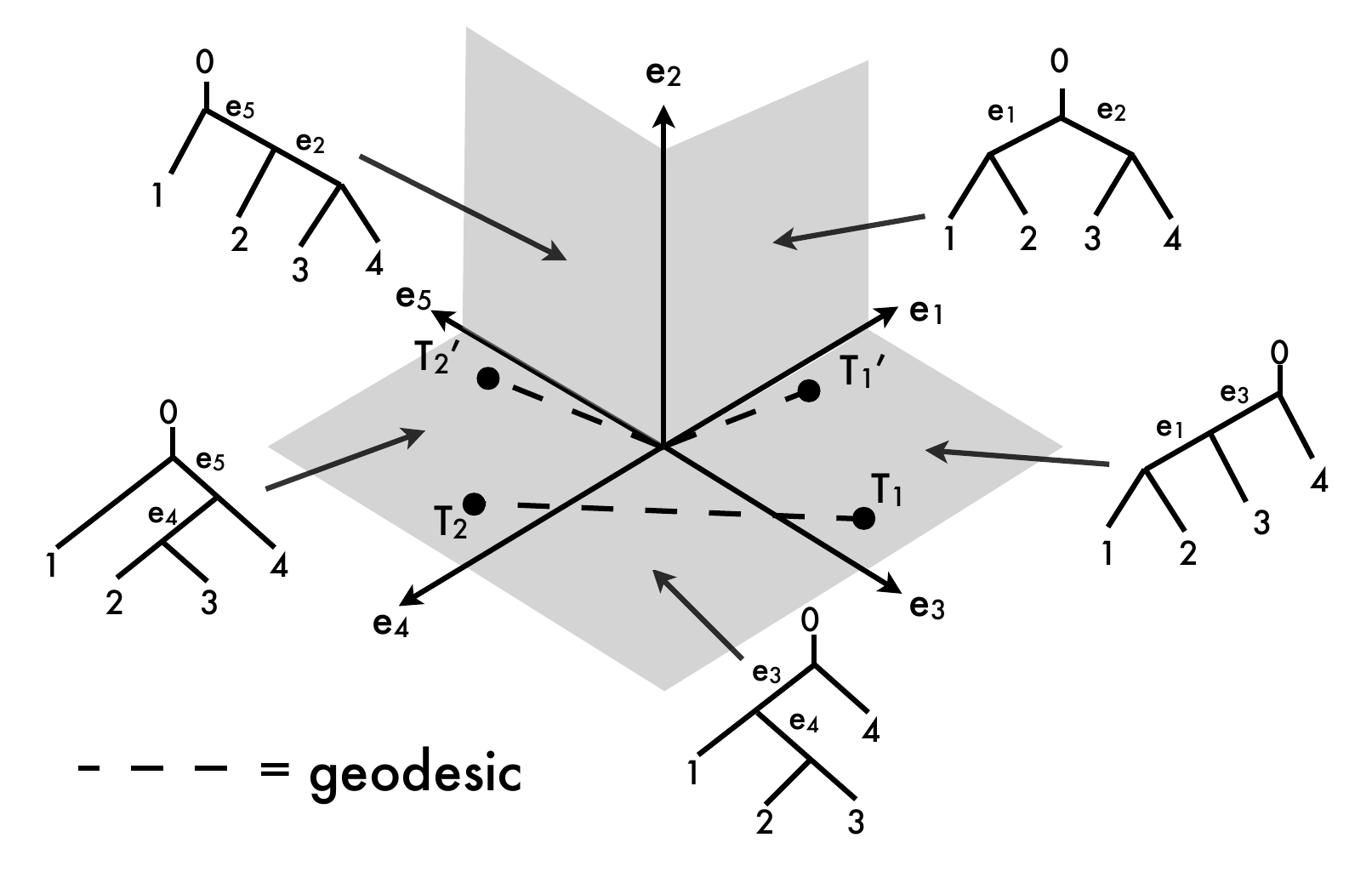}
\label{fig:5_orthants}
}
\caption{The geometry of tree space.}
\end{figure}
For any set $A$ of compatible splits with lengths, let $T(A)$ represent the tree containing exactly the edges corresponding to the splits $A$, with the given lengths.  Let $\Or(A)$ be the orthant of lowest dimension containing $T(A)$.  For any $t \geq 0$, let $t \cdot A$ be the set of splits $A$ whose lengths have all been multiplied by $t$.  If $A$ and $B$ are two compatible sets of mutually compatible splits of type $n$, then we define the binary operator $+$ on the orthants of $\T_n$ by $\Or(A) + \Or(B) = \Or(A \cup B)$.
%
%
%
%
\subsection{Geodesic Distance}
There is a natural metric on $\T_n$.  The distance between two trees in the same orthant is the Euclidean distance between them.  The distance between two trees in different orthants is the length of the shortest path between them, where the length of a path is the sum of the Euclidean lengths of the intersections of this path with each orthant.
 For any trees $T_1$ and $T_2$ in $\T_n$, the \emph{geodesic distance}, $d(T_1, T_2)$, between $T_1$ and $T_2$ is the length of the \emph{geodesic}, or locally shortest path, between $T_1$ and $T_2$ in $\T_n$.  Billera et al. defined this distance, and proved that $\T_n$ is non-positively curved \cite{BH99}, and in particular CAT(0) \cite[Lemma~4.1]{BHV01}, and thus the geodesic between any two trees in $\T_n$ is unique.
 
For example, in Figure~\ref{fig:orthants}, the geodesic between the trees $T_1$ and $T_2$ is represented by the dashed line.  Figure~\ref{fig:5_orthants} depicts 5 of the 15 orthants in $\T_4$.  This figure also illustrates that the edge lengths, in addition to the tree topologies, determine the intermediate orthants through which the geodesic passes. 
\subsection{The Essential Problem}
The problem of finding the geodesic between two arbitrary trees in $\T_n$ can be reduced in polynomial time to the problem of finding the geodesic between two trees with no splits in common.  Furthermore, the lengths of the pendant edges can easily be included in the distance calculation, if desired.

Vogtmann \cite{V07} proved the following theorem, which explains how to decompose the problem of finding the geodesic when the trees share a common split.  An alternative proof is given in \cite{thesis}.  Let $T_1$ and $T_2$ be two trees with a common split $e = X | Y$, where $0 \in X$, as shown in Figure~\ref{fig:common_edge_T1}.  For $i \in \{1, 2\}$, let $T_i^X$ be the tree $T_i$ with edge $e$ and any edge below $e$ contracted.  That is, any edge $e' = X' | Y'$ such that $X' \subset Y$ or $Y' \subset Y$ is contracted, as shown in Figure~\ref{fig:common_edge_T1A}.  For $i \in \{1, 2\}$, let $T_i^Y$ be the tree $T_i$ formed by contracting edge $e$ and all edges not contracted in $T_i^X$.  That is, any edge $e' = X' | Y'$ such that $X' \subset X$ or $Y' \subset X$ is contracted, as in Figure~\ref{fig:common_edge_T1B}.     
\begin{figure}[ht]
\centering
\subfigure[Tree $T_i$.]{
\includegraphics[scale=0.3]{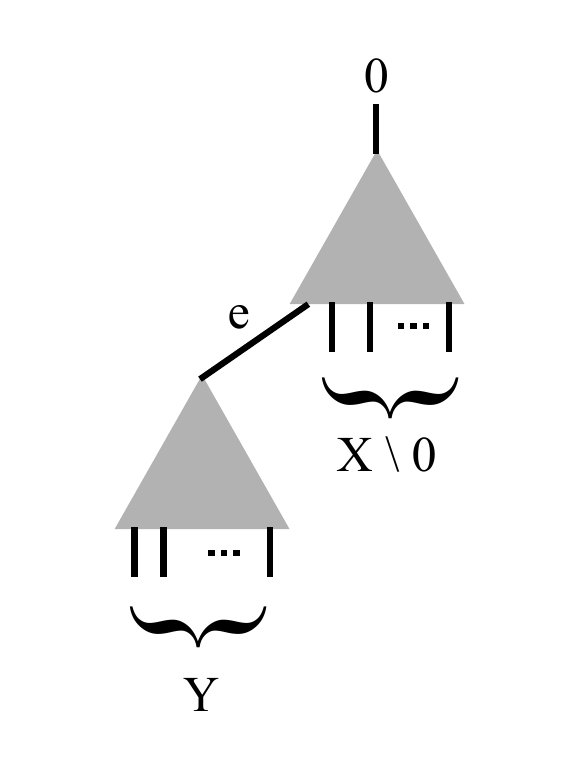}
\label{fig:common_edge_T1}
}
\subfigure[Tree $T_i^X$.]{
\includegraphics[scale=0.3]{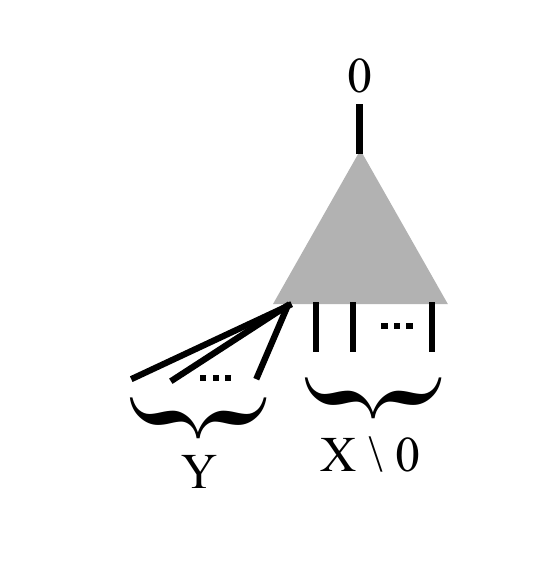}
\label{fig:common_edge_T1A}
}
\subfigure[Tree $T_i^Y$.]{
\includegraphics[scale=0.3]{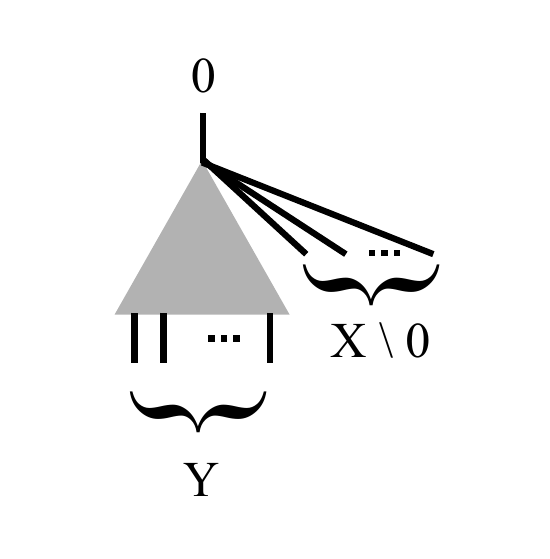}
\label{fig:common_edge_T1B}
}
\caption{Forming the trees $T_i^X$ and $T_i^Y$ from $T_i$ for $i \in \{1, 2\}$.}
\label{fig:common_edge}
\end{figure}
\begin{thm}{} 
\label{th:dist_if_common_edge}
If $T_1$ and $T_2$ have a common split $e$, and $T_i^X$ and $T_i^Y$ are as described in the above paragraph for $i \in \{1, 2\}$, then $d(T_1, T_2) = \sqrt{d(T_1^X, T_2^X)^2 + d(T_1^Y, T_2^Y)^2 + \left ( \abs{e}_{T_1} - \abs{e}_{T_2} \right )^2}$. 
\end{thm} 
As noted in Section~\ref{sec:tree_space}, the length of the edges ending in leaves can be included in the distance calculations by considering the product space $\T_n \times \R^n_+$, and the shortest distance, $d_l(T_1, T_2)$, between the trees in this space.  In this case, if the length of the edge to leaf $i$ in tree $T$ is $\abs{l_i}_{T}$ for all $1 \leq i \leq n$, then $d_l(T_1, T_2) = \sqrt{d(T_1, T_2)^2 + \sum_{i=1}^n\left( \abs{l_i}_{T_1} - \abs{l_i}_{T_2} \right)^2}$.

Therefore, the essential problem is as follows, and we devote the rest of this paper to it.
\begin{prob}
Find the geodesic distance between $T_1$ and $T_2$, two trees in $\T_n$ with no common splits.
\end{prob}
%
%
%
%
\section{Combinatorics of Path Spaces}
The properties of the geodesic imply that it is restricted to certain orthants in the tree space.  In this section, we model this section of tree space as a partially ordered set (poset), called the \emph{path poset}, in which each element corresponds to an orthant in tree space.  This poset enables us to enumerate all orthant sequences that could contain the geodesic, because each such orthant sequence, called a \emph{path space}, corresponds to one of the maximal chains of this poset by Theorem~\ref{th:one_to_one}.  

For this section, assume that $T_1 = (X, \Sigma_1)$ and $T_2 = (X, \Sigma_2)$ are two trees in $\T_n$ with no common splits.  That is, $\Sigma_1 \cap \Sigma_2 = \varnothing$.
\subsection{The Incompatibility and Path Partially Ordered Sets}
We first define the incompatibility poset, which encodes the incompatibilities between splits in $T_1$ and $T_2$.  It will be used to construct the path poset.  To define these posets, we introduce the following two definitions.

Let $A$ and $B$ be two sets of mutually compatible splits of type $n$, such that $A \cap B = \varnothing$.  Define the \emph{compatibility set of $A$ in $B$}, $C_B(A)$, to be the set of splits in $B$ which are compatible with every split in $A$.  Define the \emph{crossing set of $A$ in $B$}, $X_{B}(A)$, to be the set of splits in $B$ which are incompatible with at least one split in $A$.  


If $D$ is a set of mutually compatible splits of type $n$ such that $D \subseteq A$, then:
\begin{my_enumerate}
\item $C_{B}(A) \subseteq C_{B}(D)$ (\emph{opposite monotonicity of the compatibility set}),
\item $X_B(D) \subseteq X_B(A)$ (\emph{monotonicity of the crossing set}),
\item $C_B(A)$ and $X_B(A)$ partition $B$ (\emph{partitioning}).
\end{my_enumerate}

A \emph{preposet} or \emph{quasi-ordered set} is a set $P$ and binary relation $\leq$ that is reflexive and transitive.  See \cite[Exercise 1]{Stanley1} for more details.  Define the \emph{incompatibility preposet}, $\widetilde{P}(\Sigma_1, \Sigma_2)$, to be the preposet containing the elements of $\Sigma_2$, ordered by inclusion of their crossing sets.  So, for any $f, f' \in \Sigma_2$, $f \leq f'$ in $\widetilde{P}(\Sigma_1, \Sigma_2)$ if and only if $X_{\Sigma_1}(f) \subseteq X_{\Sigma_1}(f')$.  Define the equivalence relation $f \sim f'$ if and only if $f \leq f'$ and $f' \leq f$.  Thus, all the splits in an equivalence class have the same crossing set, which we define to be the crossing set of that equivalence class.  
\begin{defn}
\label{defn:incompatibility_poset}
The \emph{incompatibility poset}, $P(\Sigma_1, \Sigma_2)$, consists of the equivalence classes defined by $\sim$ in the preposet $\widetilde{P}(\Sigma_1, \Sigma_2)$ ordered by inclusion of their crossing sets.  
\end{defn}
Generally, we will be informal, and treat the elements of the incompatibility poset as sets of $\Sigma_2$, ordered by inclusion of their crossing sets in $\Sigma_1$.  For example, Figure~\ref{fig:DPexampleIncompPoset} shows the incompatibility poset $P(\Sigma_1, \Sigma_2)$ for the trees $T_1$ and $T_2$, given in Figures~\ref{fig:DPexampleT1} and~\ref{fig:DPexampleT2}, respectively. 
%
\begin{figure}[ht]
\centering
\subfigure[Tree $T_1 = (X, \Sigma_1)$.]{
\includegraphics[scale=0.5]{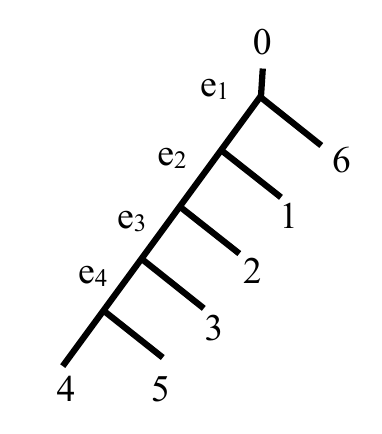}
\label{fig:DPexampleT1}
}
\subfigure[Tree $T_2 = (X, \Sigma_2)$.]{
\includegraphics[scale=0.5]{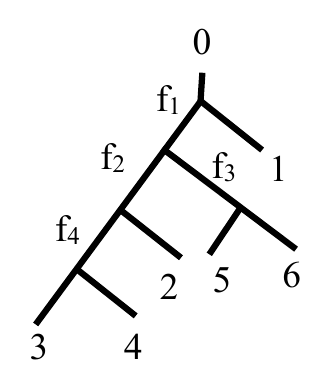}
\label{fig:DPexampleT2}
}
\subfigure[Incompatibility poset $P(\Sigma_1, \Sigma_2)$]{
\includegraphics[scale=0.5]{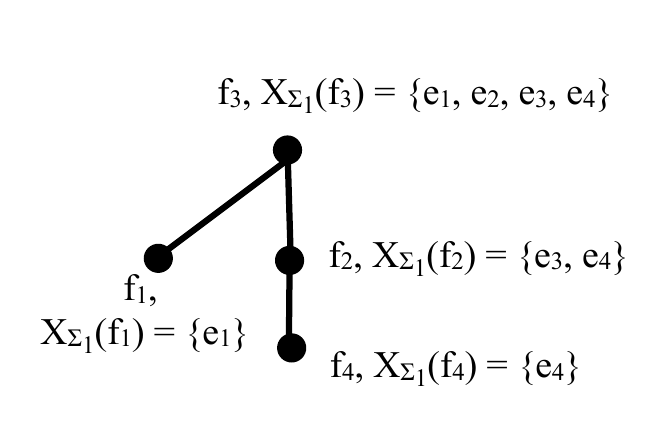}
\label{fig:DPexampleIncompPoset}
}
\subfigure[Path poset $K(\Sigma_1, \Sigma_2)$]{
\includegraphics[scale=0.3]{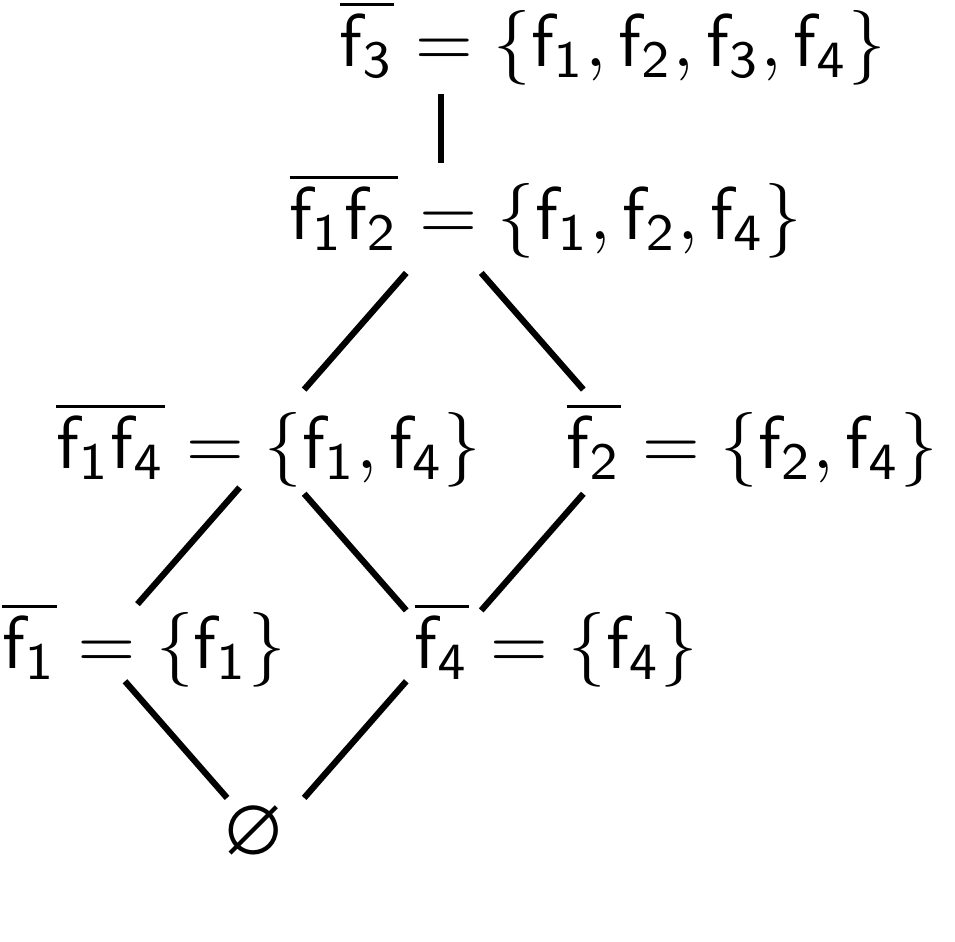}
\label{fig:DPexamplePathPoset}
}
\caption{The incompatibility poset for the trees $T_1$ (a) and $T_2$ (b) is shown in (c).  The crossing sets of the elements of $\Sigma_2$, which are ordered by inclusion to give the incompatibility poset, are also shown in the labels.  The path poset of $T_1$ and $T_2$ is given in (d).}
\label{fig:basic_example}
\end{figure}

For any $A \in \Sigma_2$, define $\overline{A} \in \Sigma_2$ by 
\[A \mapsto \overline{A} = \{f \in \Sigma_2 : X_{\Sigma_1}(f) \subseteq X_{\Sigma_1}(A) \}. \]
Note that by definition, $X_{\Sigma_1}(A) = X_{\Sigma_1}(\overline{A})$.  The map $X \mapsto \overline{X}$ is a \emph{closure operator} on a set $I$ if for every subset $X \subset I$, $\overline{ }$ it is extensive ($X \subset \overline{X}$), idempotent ($\overline{X} = \overline{\overline{X}}$), and isotone (if $X \subset Y$, then $\overline{X} \subset \overline{Y}$) \cite{Birkhoff67}.  From the definition and the monotonicity of crossing set, $A \mapsto \overline{A}$ is a closure operator on $\Sigma_2$.  
%
\begin{defn}
The \emph{path poset from $\Sigma_1$ to $\Sigma_2$}, $K(\Sigma_1,\Sigma_2)$, is the closed sets of $\Sigma_2$ ordered by inclusion.
\end{defn}
The path poset represents the possible orthant sequences containing the geodesic between $T_1$ and $T_2$, and we next make clear this correspondence.  The path poset is bounded below by $\varnothing$, and above by $\Sigma_2$.  It is a sublattice of the lattice of order ideals of $P(\Sigma_1, \Sigma_2)$, but need not be graded \cite{thesis}.  Figure~\ref{fig:DPexamplePathPoset} gives an example of a path poset.  For simplicity in the figures, we omit the brackets, writing $\overline{f_1 f_4}$ instead of $\overline{ \{f_1, f_4 \} }$, for example.
%
%
%
%
%
%
%
%
%
\subsection{Path Spaces}
The geodesic is contained in some sequence of orthants connecting the orthants containing $T_1$ and $T_2$.  Billera et al. \cite{BHV01} defined a set of orthant sequences, such that at least one of them contains the geodesic.  We call such orthant sequences \emph{path spaces}.  We characterize all maximal path spaces in Theorem~\ref{th:max_path_space_conditions}, and show that they are in one-to-one correspondence with the maximal chains in $K(\Sigma_1, \Sigma_2)$ in Theorem~\ref{th:one_to_one}.  
%
%
\begin{defn}
For trees $T_1$ and $T_2$ with no common splits, let $\Sigma_1 = E_0 \supset E_1 \supset ... \supset E_{k-1} \supset  E_k = \varnothing$, and $\varnothing = F_0 \subset F_1 \subset ... \subset F_{k-1} \subset F_k = \Sigma_2$ be sets of splits such that $E_i$ and $F_i$ are compatible for all $0 \leq i \leq k$.  Then $\cup_{i=0}^k \mathcal{O}(E_i \cup F_i)$ is a \emph{path space} between $T_1$ and $T_2$.
\end{defn}

A path space is a subspace of $\T_n$ consisting of the closed orthants corresponding to the trees with interior edges $E_i \cup F_i$ for all $0 \leq i \leq k$.  The intersection of $\Or_i$ and $\Or_{i+1}$ is the orthant $\mathcal{O}(E_{i+1} \cup F_i)$.  If the $i^{th}$ step transforms the tree with splits $E_{i-1} \cup F_{i-1}$ into the tree with splits $E_i \cup F_i$, then at this step we remove the splits $A_i \triangleq E_{i-1} \backslash E_i$ and add the splits $B_i \triangleq F_i \backslash F_{i-1}$.  Using this notation, the $i$-th orthant corresponds to the splits $B_1 \cup ... \cup B_i \cup A_{i+1} \cup .... \cup A_k$.  To simplify notation, let $\Or_i = \Or(E_i \cup F_i)$ and $\Or_i' = \Or(E_i' \cup F_i')$.

The following property of path spaces follows directly from the definition.

%
%
\begin{prop}{}
\label{prop:path_space_property}
Let $\cup_{i=0}^k \mathcal{O}(E_i \cup F_i)$ be a path space between $T_1$ and $T_2$.  Then $E_i \subseteq C_{\Sigma_1}(F_i)$ and $F_i \subseteq C_{\Sigma_2}(E_i)$ for all $0 \leq i \leq k$.
\end{prop}

%
%

\begin{rem}{}
\label{rem:equality_in_path_spaces}
In order to ensure a unique representation of a path space in terms of $E_i$'s and $F_i$'s, we make the inclusions strict in the definition of a path space.  However, if we have sets of splits $\Sigma_1 = E_0 \supseteq E_1 \supseteq \cdots \supseteq E_{k-1} \supseteq  E_k = \varnothing$ and $\varnothing = F_0 \subseteq F_1 \subseteq \cdots \subseteq F_{k-1} \subseteq F_k = \Sigma_2$ such that $E_i$ and $F_i$ are compatible for all $0 \leq i \leq k$, then $\cup_{i=0}^k \mathcal{O}(E_i \cup F_i)$ can be represented by some $\cup_{i = 0}^{k'}\Or_i'$ such that $\Sigma_1 = E_0' \supset E_1' \supset \cdots \supset E_{k'-1}' \supset  E_{k'}' = \varnothing$ and $\varnothing = F_0' \subset F_1' \subset \cdots \subset F_{k'-1}' \subset F_{k'} = \Sigma_2$.  To do this, we group consecutive $E_i$'s and $F_i$'s into larger sets that are still mutually compatible with each other, until we have a path space.
\end{rem}

A path space is \emph{maximal} if it is not contained in any other path space.  Since \cite[Proposition 4.1]{BHV01} proves that the geodesic is contained in a path space, it must be contained in some maximal path space.  We now characterize the maximal path spaces using split compatibility.
%
%
 \begin{thm}{}
\label{th:max_path_space_conditions}
The maximal path spaces from $T_1$ to $T_2$ are exactly those path spaces $\cup_{i=0}^k \Or_i$ such that:
\begin{enumerate}
\item $E_i = C_{\Sigma_1}(F_i)$, for all $0 \leq i \leq k$. 
\item $F_i = C_{\Sigma_2}(E_i)$, for all $0 \leq i \leq k$.
\item for all $1 \leq i \leq k$, the set of splits $B_i$ is a minimal element in the incompatibility poset $P (A_i \cup ... \cup A_k, B_i \cup ... \cup B_k )$\end{enumerate}
\end{thm}
\begin{proof}
Let $\mathcal{M}$ be the set of path spaces described in the theorem.  We first show, by contradiction, that all path spaces in $\mathcal{M}$ are maximal.  Suppose not.  Then there exists some path space $M = \cup_{i=0}^k \Or_i \in \mathcal{M}$ that is strictly contained in another path space $S' = \cup_{i = 0}^{k'} \Or_i'$.  

If $\Or_j \subseteq \Or_l'$ for some $0 \leq j \leq k$ and some $0 \leq l' \leq k'$, then since $\Sigma_1$ and $\Sigma_2$ are disjoint, we have $E_j \subseteq E_l'$ and $F_j \subseteq F_l'$.  By Proposition~\ref{prop:path_space_property} and the opposite monotonicity of compatibility sets, $F_l' \subseteq C_{\Sigma_2}(E_l') \subseteq C_{\Sigma_2}(E_j)  = F_j$, where the last equality follows from Condition 2 on path spaces in $\mathcal{M}$.  Hence, $F_l' = F_j$.  Similarly, $E_l' \subseteq C_{\Sigma_1}(F_l') = C_{\Sigma_1}(F_j) = E_j$, where the last equality follows from Condition 1.  Therefore, $E_l' = E_j$, and hence $\Or_j = \Or_l'$. 

Therefore, every orthant of $M$ is also an orthant of $S'$, and thus $S'$ must contain at least one other orthant not in $M$.  Let $j$ be the smallest index for such an orthant.  More specifically, the orthant $\Or_{j-1}$ is in $M$ and $S'$, but $\Or'_j, \Or_{j+1}', ..., \Or_{j+l -1}'$ are not in $M$ and $\Or_j = \Or_{j+l}'$.  Then by definition of $M$ and $S'$, $B_j' \subseteq B_j$ and $A_j' \subseteq A_j$.  By Condition 3 and the definition of the incompatibility poset, $X_{A_j \cup \ldots \cup A_k}(B_j') = X_{A_j \cup \ldots \cup A_k}(B_j)$.  Therefore, $A_j' = A_j$, which implies that $\Or_j' \subseteq \Or_j$, a contradiction.


Let $S = \cup_{i=0}^k \Or_i$ be some path space that is not in $\mathcal{M}$. We will now prove that $S$ is contained in another path space, $S'$, and hence is not maximal.  Since $S \notin \mathcal{M}$, at least one of the three conditions does not hold.

Case 1: There exists a $0 \leq j \leq k$ such that $E' = C_{\Sigma_1}(F_j) \bs E_j$ is not empty. That is, Condition 1 does not hold. \\
We now construct a path space in which the splits $E'$ are dropped at the $j$-th step instead of an earlier one.  Define $S' = \cup_{i = 0}^k \Or'_i$, where
\begin{align*}
\Or_i'=
\begin{cases}
\Or_i + \Or(E') & \text{if $0 \leq i \leq j$} \\
\Or_i & \text{if $j < i \leq k$}
\end{cases}
\end{align*} 
Since we have only added dimensions to orthants in $S$ to define $S'$ and $\Or_i' \subset \Or_i + \Or(E')$, we have $S \subset S'$. It remains to show that $S'$ is a path space.  By definition, $E'$ is compatible with $F_j$, and hence  $F_0 \subset ... \subset F_{j-1} \subset F_j$, so the splits specifying each orthant of $S'$ are compatible.  Since $\Sigma_1 = E'_0 \supseteq E'_1 \supseteq ... \supseteq E'_j \supset ... \supset E'_k = \varnothing$,  then by Remark~\ref{rem:equality_in_path_spaces}, $S'$ can be relabelled as a path space and hence $S$ is not a maximal path space.

Case 2: There exists $0 \leq j \leq k$ such that $F' = C_{\Sigma_2}(E_j) \bs F_j$ is not empty.  That is, Condition 2 does not hold. \\
We will now construct a path space in which the splits $F'$ are added to the tree at the $j$-th step, instead of a later step.  Define $S' = \cup_{i = 0}^k \Or'_i$, where 
\begin{align*}
\Or_i'=
\begin{cases}
\Or_i & \text{if $0 \leq i < j$} \\
\Or_i + \Or(F') & \text{if $j \leq i \leq k$}
\end{cases}
\end{align*}

By analogous reasoning to Case 1, $S'$ is a path space strictly containing $S$, and therefore $S$ is not maximal.


Case 3:  Let $P = P(E_{j-1}, \Sigma_2 \bs F_{j-1}) = P(A_j \cup ... \cup A_k, B_j \cup ... \cup B_k)$.  Neither Case 1 nor Case 2 holds, and, for some $1 \leq j \leq k$, there exist splits $f \in B_j$ and $g \in B_j \cup ... \cup B_k$ such that $g < f$ in $P$.  That is, Conditions 1 and 2 hold, but Condition 3 does not hold.

We now construct a path space with an extra orthant, which we get by adding the splits $g$ and $f$ in two distinct steps, instead of during the same step.
Define $S' = \cup_{i = 0}^{k+1} \Or'_i$, where 
\begin{align*}
\Or_i'=
\begin{cases}
\Or_i & \text{if $0 \leq i < j$} \\
\Or \left ( E_{i-1} \bs X_{E_{i-1}}(g) \right) +  \Or \left ( \overline{F_{i-1} \cup g} \right ) & \text{if i =j} \\
\Or_{i-1} & \text{if $j < i \leq k$}
\end{cases}
\end{align*}
We will first show that $\Or_j'$ is neither contained in nor contains any orthant from $S$, by showing that $E_{j-1}' \supset E_j' \supset E_{j+1}'$ and $F_{j-1}' \subset F_j' \subset F_{j+1}$.  We must have $X_{E_{j-1}}(g) \neq \varnothing$, or else $g \in C_{\Sigma_2}(E_{j-1}) \bs F_{j-1}$, implying Case 2 holds, which is a contradiction.  This implies that $E_{j-1} \supset E_{j-1} \bs X_{E_{j-1}}(g)$, or $E'_{j-1} \supset E'_j$.  Since $g < f$ in $P$, we have $X_{E_{j-1}}(g) \subset X_{E_{j-1}}(f)$.  To add $f$ at step $j$ in $S$, we must drop all splits in $E_{j-1}$ that are incompatible with $f$, so $X_{E_{j-1}}(f) \subseteq A_j$.  Along with the previous statement, this implies that $X_{E_{j-1}}(g) \subset A_j$, and hence $E_{j+1}' \subset E'_j$.  Therefore, we have shown that $E_{j-1}' \supset E_j' \supset E_{j+1}'$, as desired.

Since $g \notin F_{j-1}$, we have $F_{j-1} \subset \overline{F_{j-1} \cup g}$, and hence $F_{j-1}' \subset F_j'$.  It now remains to show that $F_j' \subset F_{j+1}'$, which we will do by showing that $f \in F_j$ but $f \notin \overline{F_{j-1} \cup g}$.  The first statement follows because $f \in B_j = F_j \bs F_{j-1}$.  For the second statement, $g < f$ in $P$ implies $X_{E_{j-1}}(g) \subset X_{E_{j-1}}(f)$.  Since $S$ is a path space, $X_{E_{j-1}}(F_{j-1}) = \varnothing$.  Thus, $X_{E_{j-1}}(F_{j-1}) \subset X_{E_{j-1}}(g) \subset X_{E_{j-1}}(f)$, which implies that $X_{\Sigma_1}(f) \nsubseteq X_{\Sigma_1}(F_{j-1}) \cup X_{\Sigma_1}(g)$, and hence $f \notin \overline{F_{j-1} \cup g}$.  Therefore, $F_{j-1}' \subset F_j' \subset F_{j+1}$.

Finally we show that the splits in $\Or_j'$ are mutually compatible.  By the definitions, $C_{\Sigma_1}(\overline{F_{j-1} \cup g}) = C_{\Sigma_1}(F_{j-1}) \cap C_{\Sigma_1}(g) \supseteq E_{j-1} \bs X_{\Sigma_1}(g) \supseteq E_{j-1} \bs X_{E_{j-1}}(g)$, and hence the splits of $\Or_j'$ are mutually compatible.  The other orthants remain unchanged, and thus $S'$ is a path space.  Since $S'$ strictly contains $S$, the path space $S$ is not maximal.
\end{proof} 
Recall that in a poset $P$, $x < y$ is a \emph{cover relation}, or \emph{$y$ covers $x$}, if there does not exist any $z \in P$ such that $x < z < y$.  A \emph{chain} is a totally ordered subset of a poset. A chain is \emph{maximal} when no other elements from $P$ can be added to that subset.  See \cite[Chapter 3]{Stanley1} for an exposition of partially ordered sets. 
\begin{thm}{} \label{th:one_to_one}
Let $g: K(\Sigma_1, \Sigma_2) \to \T_n$ be given by $g(L) = \Or_L$, where $\Or_L = \Or( C_{\Sigma_1}(L)  \cup L)$, for any element $L \in K(\Sigma_1, \Sigma_2)$.  For any maximal chain $L_0 < L_1 < ... < L_k$ in $ K(\Sigma_1, \Sigma_2)$, define $h(L_0 < L_1 < ... < L_k) = \cup_{i=0}^k g(L_i)$.  Then $\cup_{i=0}^k g(L_i) = \cup_{i=0}^k \Or_{L_i}$ is a maximal path space and $h$ is a bijection between maximal path spaces from $T_1$ to $T_2$ and maximal chains in $K(\Sigma_1, \Sigma_2)$.
\end{thm}
%
%
\begin{proof}
The map $g$ is one-to-one, because if $L \neq L'$, then $\Or_L \neq \Or_{L'}$.  We now show that $h$ maps maximal chains in $ K(\Sigma_1, \Sigma_2)$ to maximal path spaces.

Let $\varnothing = L_0 < L_1 < ... < L_k = \Sigma_2$ be a maximal chain in $ K(\Sigma_1, \Sigma_2)$.  For every $0 \leq i \leq k$, let $F_i = L_i$ and $E_i = C_{\Sigma_1}(L_i)$.  We now show that $\cup_{i=0}^k \Or_i$ is a path space.  Since $K(\Sigma_1, \Sigma_2)$ is the closed sets of $\Sigma_2$ ordered by inclusion, $F_i \subset F_{i+1}$ for all $0 \leq i < k$.  By the monotonicity of crossing sets, $X_{\Sigma_1}(L_i) \subseteq X_{\Sigma_1}(L_{i+1})$.  If $X_{\Sigma_1}(L_i) = X_{\Sigma_1}(L_{i+1})$, then $L_{i+1} \subseteq \overline{L_i} = L_i$, since $L_i$ is a closed set.  This is a contradiction, and therefore, $X_{\Sigma_1}(L_i) \subset X_{\Sigma_1}(L_{i+1})$.  This implies that $C_{\Sigma_1}(L_i) \supset C_{\Sigma_1}(L_{i+1})$ by the partitioning property, and hence $E_i \supset E_{i+1}$ for all $0 \leq i < k$.  

Since $L_0 = \varnothing$, $E_0 = C_{\Sigma_1}(L_0) = \Sigma_1$, and since $L_k = \Sigma_2$, $E_k = C_{\Sigma_1}(L_k) = \varnothing$, or else $T_2$ would contain more than $n-2$ splits.  Finally, for all $0 \leq i \leq k $, $E_i$ is compatible with $F_i$ by definition.  Therefore, $\cup_{i=0}^k \mathcal{O}(E_i \cup F_i)$ is a path space.

We will now show that $\cup_{i=0}^k \Or_i$ satisfies the three conditions of Theorem~\ref{th:max_path_space_conditions}, and hence is maximal.  Since $E_i = C_{\Sigma_1}(F_i)$, Condition 1 is met.  By Proposition~\ref{prop:path_space_property}, $F_i \subseteq C_{\Sigma_2}(E_i)$.  We now show that  $F_i \supseteq C_{\Sigma_2}(E_i)$.  For any $f \in C_{\Sigma_2}(E_i)$, by definition of the crossing set, $X_{\Sigma_1}(f) \cap E_i = \varnothing$.  Since $X_{\Sigma_1}(L_i)$ and $C_{\Sigma_1}(L_i) = E_i$ partition $\Sigma_1$, then $X_{\Sigma_1}(f) \subseteq X_{\Sigma_1}(L_i)$.  This implies that $f \in \overline{L_i} = L_i = F_i$, and hence Condition 2 holds. 

To show Condition 3, suppose that for some $1 \leq j \leq k$, there exists $f \in B_j$ and a minimal element $g$ in $P(E_{j-1}, \Sigma_2 \bs F_{j-1})$ such that $g < f$ in $P(E_{j-1},\Sigma_2 \bs F_{j-1})$.
As shown in the proof of Theorem~\ref{th:max_path_space_conditions}, $F_{i-1} \subset \overline{F_{i-1} \cup g} \subset F_i$.  This implies that $L_{i-1}  < \overline{F_{i-1} \cup g} <  L_i$, 
and hence $L_i < L_{i-1}$ is not a cover relation, which is a contradiction.  Therefore, Condition 3 also holds, and $\cup_{i=0}^k \Or_i$ is a maximal path space.

So as claimed, if $L_0 < L_1 < ... < L_k$ is a maximal chain, then $h(L_0 < ... <L_k)$ is a maximal path space.  It remains to show that $h$ is a bijection.  For any maximal path space $\cup_{i=0}^k \Or_i$, $F_i < F_{i+1}$ is a cover relation for all $0 \leq i < k$ since for any $f \in B_i$, $\overline{F_i \cup f} = F_{i+1}$ by Condition 3 of Theorem \ref{th:max_path_space_conditions}.  This implies that $\varnothing = F_0 < F_1 < ... < F_k = \Sigma_2$ is a maximal chain in $ K(\Sigma_1, \Sigma_2)$ such that $h(F_0 < F_1 < ... < F_k) = \cup_{i=0}^k \Or_i$, and hence $h$ is onto.  We have that $h$ is one-to-one, because $g$ is one-to-one.  Therefore, $h$ is a bijection, which establishes the correspondence.
\end{proof}
\begin{figure}[ht]
\centering
\subfigure[Tree $T_1$.]{
\includegraphics[scale=0.3]{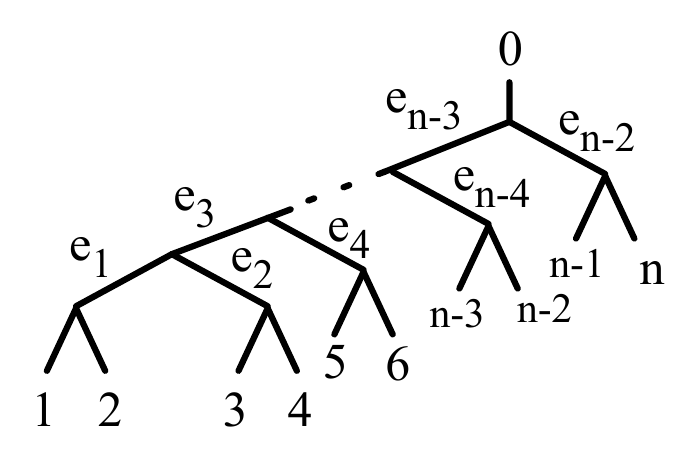}
\label{fig:dynamic_counter_ex_T1}
}
\subfigure[Tree $T_2$.]{
\includegraphics[scale=0.3]{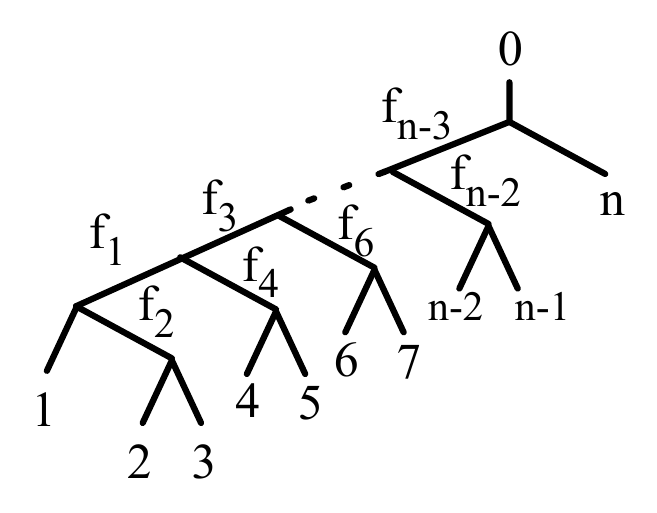}
\label{fig:dynamic_counter_ex_T2}
}
\subfigure[Incompatibility poset $P(\Sigma_1, \Sigma_2)$.]{
\includegraphics[scale=0.4]{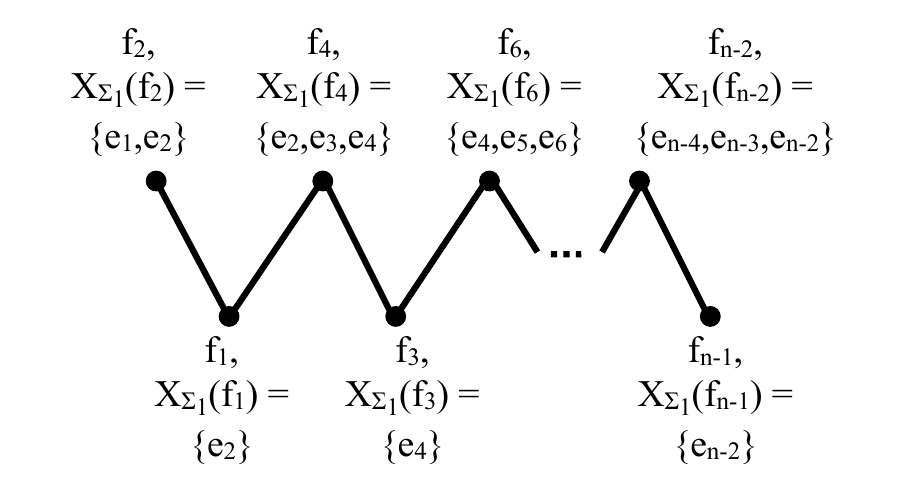}
\label{fig:dynamic_counter_ex_incomp}
}
\caption{A family of trees whose path poset is exponential in the number of leaves.}
\label{fig:dynamic_counter_ex}
\end{figure}
\begin{rem} \label{rem:exponential_path_poset}
The number of elements in a path poset $K(\Sigma_1, \Sigma_2)$ can be exponential in the number splits in the two sets.  For example, for any even positive integer $n$, consider the trees $T_1 = (X, \Sigma_1)$ and $T_2 = (X, \Sigma_2)$ depicted in Figures~\ref{fig:dynamic_counter_ex_T1} and \ref{fig:dynamic_counter_ex_T2}.  Their incompatibility poset is given in Figure~\ref{fig:dynamic_counter_ex_incomp}.  Let $W$ be the set of minimal elements in $P(\Sigma_1, \Sigma_2)$.  Then $\abs{W} = \frac{n-2}{2}$.  Each subset of $W$ is a distinct closed set, and hence an element in $K(\Sigma_1, \Sigma_2)$.  This implies there are at least $2^{(n-2)/2}$ elements in $K(\Sigma_1, \Sigma_2)$, and hence also an exponential number of maximal chains.
\end{rem}
%
%
\section{Geodesics in Path Spaces}
Given a path space, this section shows how to find the locally shortest path, or \emph{path space geodesic}, between $T_1$ and $T_2$ within that space in linear time.  We do this by transforming the problem into a \emph{Euclidean shortest-path problem with obstacles} (\cite{Mitchell98} and references) in Theorem~\ref{th:homeo_to_Rm}.  We next reformulate the problem as a touring problem \cite{DELM03}.  A \emph{touring problem} asks for the shortest path through Euclidean space that visits a sequence of regions in the prescribed order.  Lemma~\ref{lem:for_key_lem} and Lemma~\ref{lem:inductive_step_for_shortest_path_alg} give conditions on the path solving the touring problem.  The linear algorithm for computing the path space geodesic is given in Section~\ref{sec:linalg}, with Theorem~\ref{th:get_dist} proving its correctness. 
\subsection{Two Equivalent Euclidean Space Problems}
Let $T_1$ and $T_2$ be two trees with no common splits, and let $S = \cup_{i=0}^k \Or(E_i \cup F_i)$ be a path space between them.  Define the \emph{path space geodesic between $T_1$ and $T_2$ through $S$} to be the shortest path between $T_1$ and $T_2$ contained in $S$.  Let $d_S(T_1, T_2)$ be the length of this path.

We will now show that the path space geodesic between $T_1$ and $T_2$ through a path space containing $k+1$ orthants is contained in a subspace of $\T_n$ isometric to the following subset of $\R^k$.  For $0 \leq i \leq k$, define the orthant
\[V_i = \{(x_1, ..., x_k) \in \R^k : \text{$x_j \leq 0$ if $j \leq i$ and $x_j \geq 0$ if $j > i$} \}. \]
Let $V(\R^k) = \cup_{i=0}^k V_i$.

We prove three properties of path space geodesics, and hence also geodesics, in Proposition~\ref{prop:local_geo_1}, Proposition~\ref{prop:local_geo_2}, and Corollary~\ref{cor:local_geo_3}.  
These properties imply that the path space geodesic is a straight line except possibly at the intersections between orthants, where it may bend.  Furthermore, if we know the point on the path space geodesic at which an edge is added or dropped, then we know the length of that edge at any other point on the path space geodesic. 
Analogous properties were proven by Vogtmann \cite{V07} for geodesics.
%
\begin{prop}
The path space geodesic is a straight line in each orthant that it traverses.
\label{prop:local_geo_1}
\end{prop}
\begin{proof}
If not, replace the path within each orthant with a straight line, which enters and exits the orthant at the same points as the original path, to get a shorter path.
\end{proof}
%

\begin{prop}
Moving along the path space geodesic, the length of each non-zero edge changes in the trees on it at a constant rate with respect to the geodesic arc length.  That is, for any edge $e \in \Sigma_1 \cup \Sigma_2$, there exists a constant $c_e >0$ such that $\frac{\abs{e}_T}{d_S(T_1, T)} = c_e$ for any tree $T$ on the geodesic that contains edge $e$.
\label{prop:local_geo_2}
\end{prop}
\begin{proof}
By Proposition~\ref{prop:local_geo_1}, each edge must shrink or grow at a constant rate with respect to the other edges within each orthant, but these rates can differ between orthants.  That is, Proposition~\ref{prop:local_geo_1} allows the constant $c_e$ to depend on the orthant containing $T$, but we will now show that it does not.  It suffices to consider when the geodesic goes through the interiors of the two adjacent orthants $\Or_{i-1} = \Or(E_{i-1} \cup F_{i-1})$ and $\Or_i = \Or(E_i \cup F_i)$, and bends in the intersection of these two orthants.  Let $\mathbf{a}$ be the point at which the geodesic enters $\Or_{i-1}$, and let $\mathbf{b}$ be the point at which the geodesic leaves $\Or_i$.

The edges $A_i = E_{i-1} \bs E_i$ are dropped and the edges $B_i = F_i \bs F_{i-1}$ are added as the geodesic moves from $\Or_{i-1}$ to $\Or_i$.  Thus the edges $A_i$ and $B_i$ all have length 0 in the intersection $\Or(E_i \cup F_{i-1})$.  

Let $m = \abs{E_i \cup F_{i-1}}$, the dimension of $\Or_{i-1} \cap \Or_i$.  An affine hull of a set $S$ in $\R^n$ is the intersection of all affine sets containing $S$.  Consider the subset $S = H_a \cup H_b$ of $\Or_{i-1} \cup \Or_i$, where $H_a$ is the affine hull of $\mathbf{a} \cup (\Or_{i-1} \cap \Or_i)$ intersected with $\Or_{i-1}$ and $H_b$ is the affine hull of $\mathbf{b} \cup (\Or_{i-1} \cap \Or_i)$ intersected with $\Or_i$.  This subset can be isometrically mapped into two orthants in $\R^{m+1}$ as follows.  For each tree $T \in H_a$, let the first $m$ coordinates be given by the projection of $T$ onto $\Or_{i-1} \cap \Or_i$.  Let the $(m+1)$-st coordinate be the length of the projection of $T$ orthogonal to $\Or_{i-1} \cap \Or_i$.  More specifically, let the edges in $E_i \cup F_{i-1}$ be $e_1, e_2, ..., e_m$.  Then we map $T$ to the point $(\abs{e_1}_{T}, \abs{e_2}_{T}, ...., \abs{e_m}_{T}, s)$ in $\R^{m+1}$, where $s = \sqrt{ \sum_{e \in A_i } \abs{e}_{T}^2 }$.  Similarly, for each tree $T \in H_b$, let the first $m$ coordinates be given by the projection of $T$ onto $\Or_{i-1} \cap \Or_i$.  Let the $(m + 1)$-st coordinate be the negative of the length of the projection of $T$ orthogonal to $\Or_{i-1} \cap \Or_i$.  In other words, we map $T$ to the point $(\abs{e_1}_{T}, \abs{e_2}_{T}, ...., \abs{e_m}_{T}, -s)$ in $\R^{m+1}$, where $s = \sqrt{ \sum_{e \in B_i } \abs{e}_{T}^2 }$.

We have mapped $S$ into Euclidean space, and hence the shortest path between the image of $\mathbf{a}$ and the image of $\mathbf{b}$ is the straight line between them.  Along this line, each edge $e_1, ..., e_m$ changes at the same rate with respect to the geodesic arc length.  Since we can make this argument for each pair of consecutive orthants, we have proven this proposition.
\end{proof}
%
\begin{cor}
Let $T$ be a tree on the path space geodesic between $T_1$ and $T_2$ through the path space $S = \cup_{i = 0}^k \Or(E_i \cup F_i)$.  Suppose $T \in \Or_i$.  Then if $1 \leq j \leq i$, we have $\frac{ \abs{f_1}_{T} }{ \abs{f_1}_{T_2} } = \frac{ \abs{f_2}_{T} }{ \abs{f_2}_{T_2} }$ for any $f_1, f_2 \in B_j$, and if $i < j \leq k$, we have $\frac{ \abs{e_1}_{T} }{ \abs{e_1}_{T_1} } = \frac{ \abs{e_2}_{T} }{ \abs{e_2}_{T_1} }$ for any $e_1, e_2 \in A_j$.
\label{cor:local_geo_3}
\end{cor}
\begin{proof}
Let $f_1, f_2 \in B_j$ be edges in the tree $T \in \Or_i$ from the hypothesis.  Then by Proposition~\ref{prop:local_geo_2}, there exist $c_{f_1}, c_{f_2} > 0$ such that $\abs{f_1}_T = c_{f_1} \cdot d_S(T_1, T)$, $\abs{f_1}_{T_2} = c_{f_1} \cdot d_S(T_1, T_2)$, $\abs{f_2}_T = c_{f_2} \cdot d_S(T_1, T)$, and $\abs{f_2}_{T_2} = c_{f_2} \cdot d_S(T_1, T_2)$.  Then 
$\frac{ \abs{f_1}_{T} }{ \abs{f_1}_{T_2} } = \frac{ c_{f_1} \cdot d_S(T_1, T) }{ c_{f_1} \cdot d_S(T_1, T_2) } = \frac{ d_S(T_1, T) }{ d_S(T_1, T_2) } = \frac{ c_{f_2} \cdot d_S(T_1, T) }{ c_{f_2} \cdot d_S(T_1, T_2) } = \frac{ \abs{f_2}_{T} }{ \abs{f_2}_{T_2} }.$
The argument to show $\frac{ \abs{e_1}_{T} }{ \abs{e_1}_{T_1} } = \frac{ \abs{e_2}_{T} }{ \abs{e_2}_{T_1} }$ for any $e_1, e_2 \in A_j$ is analogous.

\end{proof}
Therefore, there is one degree of freedom for each set of edges dropped, or alternatively for each set of edges added, at the transition between orthants.  Thus, the path space geodesic lies in a space of dimension equal to the number of transitions between orthants.  We will now show that each path space geodesic lives in a space isometric to $V(\R^k)$.  For example, in Figure~\ref{fig:tree_to_Euclidean_space_1}, the path space $Q$ consists of the orthants $\Or(\{e_1, e_2, e_3\})$, $\Or(\{f_1, e_2, e_3\})$, and $\Or(\{f_1, f_2, f_3\})$.  We apply Theorem~\ref{th:homeo_to_Rm} to see that the geodesic through $Q$ is contained in the shaded region of $\R^2$ shown in Figure~\ref{fig:tree_to_Euclidean_space_2}.
\begin{figure}[ht]
\centering
\subfigure[Part of $\T_5$.]{
\includegraphics[scale=0.35]{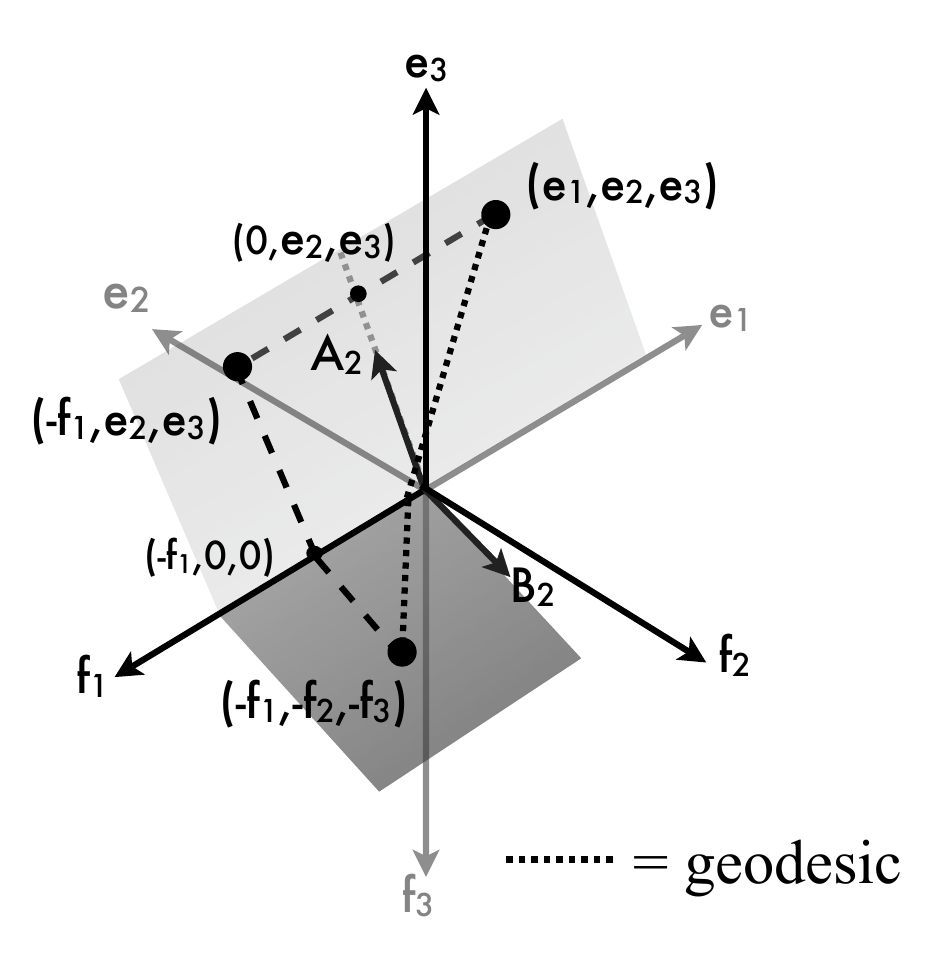}
\label{fig:tree_to_Euclidean_space_1}
}
\subfigure[Isometric mapping to $V(\R^2)$.]{
\includegraphics[scale=0.35]{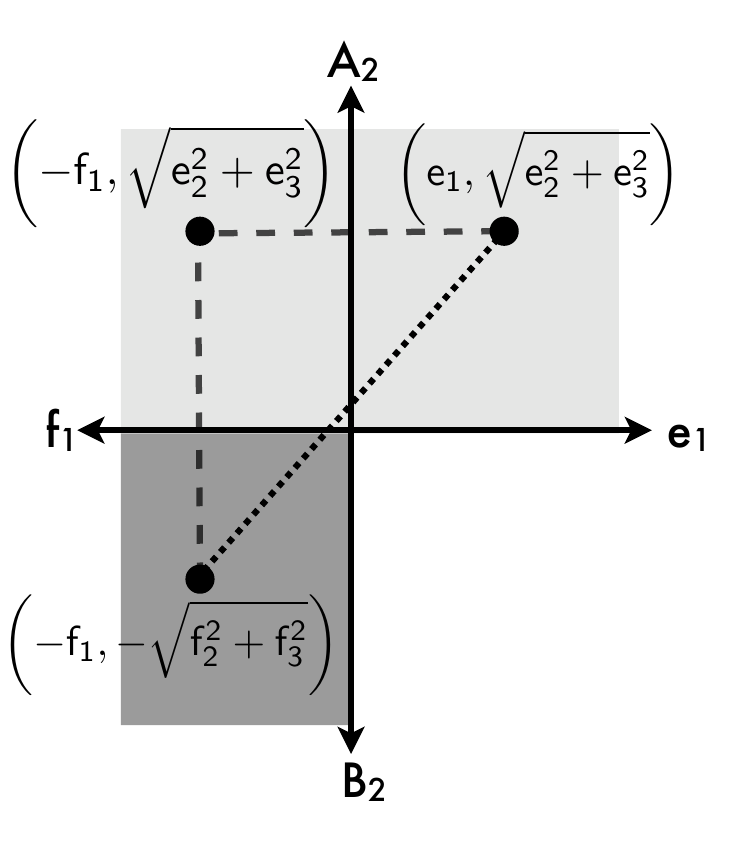}
\label{fig:tree_to_Euclidean_space_2}
}
\caption{An isometric map between a path space and $V(\R^2)$.}
\label{fig:tree_to_Euclidean_space}
\end{figure}

%
\begin{thm}{}
\label{th:homeo_to_Rm}
Let $Q = \cup_{i =0}^k \mathcal{O}(E_i \cup F_i)$ be a path space between $T_1$ and $T_2$, two trees in $\T_n$ with no common splits.  Then the path space geodesic between $T_1$ and $T_2$ through $Q$ is contained in a space isometric to $V(\R^k)$.
\end{thm}
\begin{proof}
By Corollary~\ref{cor:local_geo_3}, any tree $T' \in Q$ on the path space geodesic satisfies the following two conditions for each $1 \leq j \leq k$:
\begin{enumerate}
\item if $T' \in \Or_i$ and $j \leq i$, then there exists a $c_j = c_j(T') \geq 0$, depending on $T'$, such that $\frac{\abs{f}_{T'}}{ \abs{f}_{T_2}} = c_j$ for all $f \in B_j$,\\
\item if $T' \in \Or_i$ and $j > i$, then there exists a $d_j = d_j(T') \geq 0$, depending on $T'$, such that $\frac{\abs{e}_{T'}}{ \abs{e}_{T_1}} = d_j$ for all $e \in A_j$.
\end{enumerate}
Let $Q' \subset \T_n$ be the set of trees satisfying this property.  For $0 \leq i \leq n$, define $h_i:Q' \cap \Or_i \to V_i$ by 
\begin{align*}
h_i \bigl(T'\bigr) &= h_i \bigl(T \left (c_1 \cdot B_1 \cup ... \cup c_i \cdot B_i \cup d_{i+1}\cdot A_{i+1} \cup ... \cup d_k \cdot A_k  \right ) \bigr) \\
&= \bigl(-c_1 ||B_1||, ..., -c_i ||B_i||, d_{i+1} ||A_{i+1} ||, ..., d_k ||A_k|| \bigr).
\end{align*}
We claim that $h_i$ is a bijection from $Q' \cap \Or_i$ to the orthant $V_i$ in $V(\R^k)$.  All trees in the interior of orthant $\Or_i$ have exactly the edges $\{B_1, ..., B_i, A_{i+1}, ..., A_k\}$.  Let $N = |B_1| + |B_2| + ... + |B_i| + |A_{i+1}| + .... + |A_k|$, the number of edges in trees in $\Or_i$.  Then $\Or_i$ is an $N$-dimensional orthant, and 
 we can assign each edge to a coordinate axis so that the edges in $B_1$ are assigned to coordinates 1 to $|B_1|$, the edges in $B_2$ are assigned to coordinates $|B_1| + 1$ to $|B_1| + |B_2|$, the edges in $A_{i+1}$ are assigned to the coordinates $|B_1| + |B_2| + ... +  |B_i| + 1$ to $|B_1| + |B_2| + ... +  |B_i| + |A_{i+1}|$, etc.  Let $e_j$ be the edge assigned to the $j$-th coordinate.  By abuse of notation, for all $1 \leq j \leq i$, let $\bold{B_j}$ be the $N$-dimensional vector with a 0 in every coordinate except those corresponding to the edges $B_j$, where we put the length of that edge in $T_2$.  Similarly, for all $i < j \leq k$, let $\bold{A_j}$ be the $N$-dimensional vector with a 0 in every coordinate except those corresponding to the edges in $A_j$, where we put the length of that edges in $T_1$.  For example, $\bold{B_1}$ is the $N$-dimensional vector $(|f_1|_{T_2}, |f_2|_{T_2}, ..., |f_{|B_1|}|_{T_2}, 0, ..., 0)$.  

Then $Q' \cap \Or_i$ is generated by the vectors $\left \{ \frac{\bold{B_1}}{\norm{ \bold{B_1} } }, \frac{\bold{B_2}}{ \norm{ \bold{B_2} } }, ..., \frac{\bold{B_i}}{ \norm{ \bold{B_i} } }, \frac{ \bold{A_{i+1}} }{ \norm{ \bold{A_{i+1}} } }, .., \frac{ \bold{A_k}}{ \norm{ \bold{A_k} }} \right \}$.  Since these generating vectors are pairwise orthogonal, they are independent, and hence $Q' \cap \Or_i$ is a $k$-dimensional orthant contained in $\Or_i$.  Furthermore, for all $1 \leq j \leq i$, $\frac{ \bold{B_j} }{ \norm{ \bold{B_j} } }$ corresponds to the tree 
$T \left (\frac{1}{ \norm{ \bold{B_j} } } \cdot B_j \right),$
 and for all $i < j \leq k$, $\frac{ \bold{A_j} }{ \norm{ \bold{A_j} } }$ corresponds to the tree 
 $T \left (\frac{1}{ \norm{ \bold{A_j} } } \cdot A_j \right).$
 For all $1 \leq j \leq k$, let $\mathbf{u_j}$ be the $k$-dimensional unit vector with a 1 in the $j$-th coordinate.  Then for $1 \leq j \leq i$, 
  \begin{align*}
 h_i \left (\frac{ \bold{B_j} }{ \norm{ \bold{B_j} } } \right ) = h_i \left (T \left ( \frac{1}{ \norm{ \bold{B_j} } } \cdot B_j \right ) \right) = - \frac{1}{ \norm{ \bold{B_j} } } \cdot \norm{ \bold{B_j} } \mathbf{u_j} = -\mathbf{u_j}.
 \end{align*}
 Similarly,  for all $i < j \leq k$, 
 \begin{align*}
 h_i \left (\frac{ \bold{A_j} }{ \norm{ \bold{A_j} } } \right) = h_i \left (T \left( \frac{1}{ \norm{ \bold{A_j} } } \cdot A_j \right ) \right) = \frac{1}{ \norm{ \bold{A_j} } } \cdot \norm{ \bold{A_j} } \mathbf{u_j} = \mathbf{u_j}.
 \end{align*}
The basis of $V_i$ is $\{-\mathbf{u_1}, ..., -\mathbf{u_i}, \mathbf{u_{i+1}}, ..., \mathbf{u_k}\}$, so $h_i$ maps each basis element of $Q' \cap Q_i$ to a unique basis element of $V_i$.  Thus, $h_i$ is a linear transformation, whose corresponding matrix is the identity matrix, and hence a bijection between $Q' \cap Q_i$ and $V_i$ for all $i$.  Furthermore, since the determinant of the matrix of $h_i$ is 1, $h_i$ is also an isometry.  So $Q'$ is piecewise linearly isometric to $V(\R^k)$.

For all $0 \leq i \leq n$, the inverse of $h_i$ is $g_i:V_i \to Q'$ defined by 
$g_i(-x_1, ... -x_i, x_{i+1}, ..., x_k) = T',$
where $x_j \geq 0$ for all $1 \leq j \leq k$ and $T'$ is the tree with edges $E_i \cup F_i$ with lengths $\frac{|x_j|}{ \norm{ \bold{B_j} } }\cdot |e|_{T_2}$ if $e \in B_j$ for $1 \leq j \leq i$ and $\frac{ \abs{x_j} }{ \norm{ \bold{A_j} }}\cdot |e|_{T_1}$ if $e \in A_j$ for $i < j \leq k$.  

Notice that if $T' \in Q' \cap \Or_i \cap \Or_{i+1}$, then $h_i(T') = h_{i+1}(T')$, since the lengths of all the edges in $A_{i+1}$ and $B_{i+1}$ are 0.  Therefore, define $h: Q' \to V(\R^k)$ to be $h(T') = h_i(T')$ if $T' \in \Or_i \cap Q'$, which is well-defined.  Define $g:V(\R^k) \to Q'$ by setting
$g(-x_1, ... -x_i, x_{i+1}, ..., x_k) = g_i(-x_1, ... -x_i, x_{i+1}, ..., x_k),$
for all $1 \leq i \leq k$ and for all $x_j \geq 0$ for all $1 \leq j \leq k$.  Then $g$ is also well-defined and the inverse of $h$.  

For any geodesic $q$ in $Q'$, map it into $V(\R^k)$ by applying $h$ to each point on $q$ to get path $p$.  Notice that since both $h_i$ and $g_i$ are distance preserving, $p$ is the same length as $q$.  We claim $p$ is a geodesic in $V(\R^k)$.  To prove this, suppose not.  Let $p'$ be the geodesic in $V(\R^k)$ between the same endpoints as path $p$.  Then $p'$ is strictly shorter than $p$.  Use $g$ to map $p'$ back to $Q'$ to get $q'$.   Again distance is preserved, so $q'$ is strictly shorter than $q$.  But $q$ was a geodesic, and hence the shortest path between those two endpoints in $Q'$, so we have a contradiction.  Therefore, the geodesic between $T_1$ and $T_2$ in $Q$ is isometric to the geodesic between $A = ( \norm{ \bold{A_1} } , ..., \norm{ \bold{A_k} } )$ and $B = (- \norm{ \bold{B_1} } , ..., - \norm{ \bold{B_k} } )$ in $V(\R^k)$.
\end{proof}
Thus, finding the geodesic through a $(k+1)-$orthant path space $\cup_{i =0}^k \mathcal{O}(E_i \cup F_i)$ is equivalent to finding the geodesic through $V(\R^k)$ between the point $A = ( \norm{ \bold{A_1} } , ..., \norm{ \bold{A_k} } )$ and the point $B = (- \norm{ \bold{B_1} } , ..., - \norm{ \bold{B_k} } )$.  Now consider the Euclidean space $\R^k$ in which every orthant that is not in $V(\R^k)$ is replaced by an obstacle. Then finding the shortest path from $A$ to $B$ in this new space with obstacles will give us the path space geodesic in tree space.   

We will now generalize, and somewhat abuse notation, by letting $A$ be any point in the all-positive orthant of $\R^k$ and by letting $B$ be any point in the all-negative orthant of $\R^k$.  Then we can reformulate this general problem as the following touring problem:

\begin{prob}[Touring] \label{prob:touring}
Let $A$ be any point in the positive orthant of $\R^k$ and let $B$ be any point in the negative orthant of $\R^k$.  Let $P_i$ be the boundary between the $i$-th and $(i+1)$-st orthants in $V(\R^k)$, for all $1 \leq i \leq k$.  That is, 
\[P_i = \{ (x_1, ..., x_k) \in \R^k: x_j \leq 0 \text{ if $j < i$; } x_j = 0 \text{ if $j = i$; }  x_j \geq 0 \text{ if $j > i$} \}.\]
Find the shortest path between $A$ and $B$ in $\R^k$ that intersects $P_1, P_2, ..., P_k$ in that order.
\end{prob}

In dimensions 3 and higher, the Euclidean shortest path problem with obstacles is NP-hard in general \cite{CR87}, including when the obstacles are disjoint axis-aligned boxes \cite{MS04}.  The touring problem can be solved in polynomial time as a second order cone problem when the regions are polyhedra \cite{PM05}.   In the special case of the above touring problem, we find a simple linear algorithm.
%
%
\subsection{Touring Problem Solution}
In this section, we give a solution to Problem~\ref{prob:touring}.  Since this is a convex optimization problem, this solution is unique \cite{PM05}, and we will call it the shortest, ordered path.    As in the problem statement, let $A = (a_1, ..., a_k)$, where $a_i \geq 0$ for all $1 \leq i \leq k$, and let $B = (-b_1, ..., -b_k)$, where $b_i \geq 0$ for all $1 \leq i \leq k$.  First, Lemma~\ref{lem:ascending_ratio_seq_means_geo_is_line} establishes when a straight line from $A$ to $B$ passes through the regions in the desired order.  Two further properties of the shortest, ordered path are given in Lemmas~\ref{lem:for_key_lem} and \ref{lem:inductive_step_for_shortest_path_alg}.  Theorem~\ref{th:get_dist} shows how exploiting this last property, in conjunction with using Theorem~\ref{th:homeo_to_Rm} to reduce the dimension of the problem, gives a linear algorithm for finding the shortest, ordered path from $A$ to $B$.
%
%
\begin{lem}{}
\label{lem:ascending_ratio_seq_means_geo_is_line}
The line from $A$ to $B$, $\overline{AB}$, passes through the regions $P_1, P_2, ..., P_k$ in that order and has length $\sqrt{\sum_{i = 1}^k (a_i + b_i)^2}$ if and only if $\frac{a_1}{b_1} \leq \frac{a_2}{b_2} \leq ... \leq \frac{a_k}{b_k}$.
\end{lem}
\begin{proof}
Parametrize the line $\overline{AB}$ with respect to the variable $t$, so that $t=0$ at $A$ and $t=1$ at $B$, to get $(x_1, ..., x_k) = (a_1, ..., a_k) + t(-a_1 - b_1, ..., -a_k - b_k)$.  Let $t_i$ be the value of $t$ at the intersection of $\overline{AB}$ and $P_i$.  Setting $x_i =0$, and solving for $t$ gives $t_i = \frac{a_i}{a_i + b_i}$.  For $\overline{AB}$ to cross $P_1, P_2, ..., P_k$ in that order, we need $t_1 \leq t_2 \leq ... \leq t_k$ or $\frac{a_1}{a_1 + b_1} \leq \frac{a_2}{a_2 + b_2} \leq ... \leq \frac{a_k}{a_k + b_k}$.  Since for any $1 \leq i, j \leq k$, $\frac{a_i}{a_i + b_i} \leq \frac{a_j}{a_j + b_j}$ is equivalent to $\frac{a_i}{b_i} \leq \frac{a_j}{b_j}$ by cross multiplication, we get the desired condition.  By the Euclidean distance formula, the length $\overline{AB}$ is $\sqrt{\sum_{i = 1}^k (a_i + b_i)^2}$.
\end{proof}
\begin{cor}{}
\label{cor:equal_ratios}
Let $A = (a_1, ..., a_k)$ and $B = (-b_1, ..., -b_k)$ be points in $\R^k$ with $a_i, b_i \geq 0$ for all $1 \leq i \leq k$.   Then $\frac{a_i}{b_i} = \frac{a_{i+1}}{b_{i+1}}$ if and only if $\overline{AB}$ intersects $P_i \cap P_{i+1}$.
\end{cor}
\begin{proof}
This follows directly from the proof of Lemma~\ref{lem:ascending_ratio_seq_means_geo_is_line}.
\end{proof}
In general, we will not have $\frac{a_1}{b_1} \leq \frac{a_2}{b_2} \leq ... \leq \frac{a_k}{b_k}$, and hence the shortest path is not a straight line.  Since the shortest, ordered path corresponds to a path space geodesic in the shortest Euclidean path with obstacles problem, Proposition~\ref{prop:local_geo_1}, Proposition~\ref{prop:local_geo_2}, and Corollary~\ref{cor:local_geo_3} also hold here.  Therefore, the shortest, ordered path intersects each region $P_i$ at a unique point $p_i$, where the path may bend.  The path is a straight line from $p_i$ to $p_{i+1}$ for $1 \leq i <k$.  We can straighten a bend in the path by isometrically mapping the problem to a lower dimensional space using the following Corollary~\ref{cor:homeo_to_Rm} to Theorem~\ref{th:homeo_to_Rm}.  We repeat this process for each successive bend until Lemma~\ref{lem:ascending_ratio_seq_means_geo_is_line} applies.


%
\begin{cor}
\label{cor:homeo_to_Rm}
Consider the shortest path from $A= (a_1, a_2, ..., a_k)$ to $B = (-b_1, -b_2,$ $ ..., -b_k)$ in $\R^k$ passing through $P_1$, ..., $P_k$ in that order.  Let $\{M_j\}_{j= 1}^m$ be any ordered partition of $\{1, 2, ..., k\}$ such that $i,l \in M_j$ implies $p_i = p_l$.  Then this path is contained in a region of $\R^k$ isometric to $V(\R^m)$.
\end{cor}
\begin{proof}
Suppose $i, i+1$ are in the same block in $\{M_j\}_{j= 1}^m$.  Then $p_i = p_{i+1}$, and travelling along the pre-image of the path in tree space, the tree loses splits $A_i$ and $A_{i+1}$ simultaneously, and gains splits $B_i$ and $B_{i+1}$ simultaneously.  Hence, this path is in the path space $S = \Or_0 \cup \left (\cup_{j=1}^m \Or\left ( (\cap_{i \in M_j} E_i  ) \cup (\cup_{i \in M_j} F_i ) \right ) \right)$.  Apply Theorem~\ref{th:homeo_to_Rm} to $S$ to see that its path space geodesic is contained in a region isometric to $V(\R^m)$, as desired.
\end{proof}
Notice that under the mapping to $V(\R^m)$ described in the above proof, $A$ is mapped to $\widetilde{A} = \left ( \sqrt{ \sum_{i \in M_1} a_i^2}, \sqrt{ \sum_{i \in M_2} a_i^2}, ..., \sqrt{ \sum_{i \in M_m} a_i^2} \right )$ and $B$ is mapped to $\widetilde{B} = \left ( -\sqrt{ \sum_{i \in M_1} b_i^2}, -\sqrt{ \sum_{i \in M_2} b_i^2}, ..., -\sqrt{ \sum_{i \in M_m} b_i^2} \right )$.

To apply Corollary~\ref{cor:homeo_to_Rm}, we need to know when $p_i = p_{i+1}$.  A condition for this is given in Lemma~\ref{lem:inductive_step_for_shortest_path_alg}.  The following Lemma~\ref{lem:for_key_lem} is used in proving Lemma~\ref{lem:inductive_step_for_shortest_path_alg}, but it also shows that the shortest path only bends at the intersection of two or more $P_i$'s (by setting $J=i$).
%
%
\begin{lem}{}
\label{lem:for_key_lem}
Let $q$ be the shortest path from $A$ to $B$ that passes through $P_1, P_2, ..., P_k$ in that order.  Let $p_j$ be the intersection of $q$ and $P_j$ for each $1 \leq j \leq k$.  If $\frac{a_J}{b_J} \leq \frac{a_{J+1}}{b_{J+1}} \leq ... \leq \frac{a_i}{b_i}$, for some $1 \leq J \leq i < k$, $q$ is a straight line until it bends at $p_J = p_{J+1} = ... = p_i$, and $p_{J-1} \neq p_J$ if $J >1$, then $p_i = p_{i+1}$.
\end{lem}
\begin{proof}
This proof is by contradiction, so assume that $p_i \neq p_{i+1}$.  Since $q$ is a shortest, ordered path, $q$ is a straight line from $p_i$ to $p_{i+1}$.  Let $Y = (-y_1, ..., -y_i, y_{i+1}, ..., y_k)$, where $y_j \geq 0$ for all $1 \leq j \leq k$, be a point on the line $\overline{p_i p_{i+1}}$, $\varepsilon > 0$ past $p_i$.  Note that $AYp_J$ forms a non-trivial triangle, since $q$ bends at $p_J$.  We will now show that $\overline{AY}$ intersects $P_1$, $P_2$, ..., $P_i$ in that order.   

Parametrize the paths $q$ and $\overline{AY}$ with respect to time $t$, so that $t = 0$ at $A$ and $t = 1$ at $Y$.  The $j$-th coordinate, for $1 \leq j \leq J-1$, decreases linearly from $a_j$ to $-y_j$ in both $q$ and $\overline{AY}$, and thus become 0 at the same time in both paths.  This implies that since $q$ crosses $P_1, ..., P_{J-1}$ in that order, $\overline{AY}$ also crosses $P_1, ..., P_{J-1}$ in that order.

Let $t_j$ be the time at which $\overline{AY}$ intersects $P_j$, for $1 \leq j \leq i$.  Then $0 = a_j + t_j (-y_j - a_j)$ or $t_j = \frac{a_j}{y_j + a_j}$.  In $q$, each coordinate between $J$ and $i$ becomes 0 at the same time.  These coordinates then decrease linearly, so the ratio between any two consecutive coordinates remains constant as time increases.  This implies $\frac{y_j}{y_{j+1}} = \frac{b_j}{b_{j+1}}$ for each $J \leq j \leq i$.  Since $\frac{a_J}{b_J} \leq \frac{a_{J+1}}{b_{J+1}} \leq ... \leq \frac{a_i}{b_i}$ by the hypothesis, then $\frac{a_J}{y_J} \leq \frac{a_{J+1}}{y_{J+1}} \leq ... \leq \frac{a_i}{y_i}$.  This implies $\frac{a_J}{a_J + y_J} \leq \frac{a_{J+1}}{a_{J+1} + y_{J+1}} \leq ... \leq \frac{a_i}{a_i + y_i}$, or $t_J \leq t_{J+1} \leq ... \leq t_i$.  Thus $\overline{AY}$ intersects $P_J, P_{J+1}, ..., P_i$ in that order.

It remains to show that $\overline{AY}$ intersects $P_{J-1}$ before $P_J$ if $J >1$, which we do by contradiction.  So assume that $t_J < t_{J-1}$.  Let $r_{J-1}$ and $r_J$ be the points of intersection of $\overline{AY}$ with $P_{J-1}$ and $P_J$, respectively.  By the hypotheses and assumption, $r_J$ and $p_J$ are contained in $P_J \bs P_{J-1}$.  Since $P_{J-1}$ and $P_J$ are convex, $\overline{r_{J-1} p_{J-1}}$ and $\overline{r_J p_J}$ are contained in $P_{J-1}$ and $P_J$, respectively.  Now $\overline{r_{J-1} p_{J-1}}$ intersects $\overline{r_J p_J}$ inside the triangle $AYp_J$. This implies that $\overline{r_J p_J}$ passes from $P_J \bs P_{J-1}$ into $P_{J-1} \cap P_J$, on the boundary of $P_J$, and back into $P_J \bs P_{J-1}$.  But this contradicts the convexity of $P_J$.  Thus $t_{J-1} \leq t_J$, and $\overline{AY}$ passes through $P_1$, $P_2$, ..., $P_i$ in that order.

By the triangle inequality, $\overline{AY}$ is shorter than the section of $q$ from $A$ to $Y$.  This contradicts $q$ being the shortest, ordered path, and thus $p_i = p_{i+1}$.
\end{proof}
\begin{lem}{}
\label{lem:inductive_step_for_shortest_path_alg}
For the shortest path $q$ from $A$ to $B$ that passes through $P_1, P_2, ..., P_k$ in that order, if $\frac{a_1}{b_1} \leq \frac{a_2}{b_2} \leq ... \leq \frac{a_i}{b_i} > \frac{a_{i+1}}{b_{i+1}}$, then this path intersects $P_i \cap P_{i+1}$.
\end{lem}
\begin{proof}
Parametrize $q$ with respect to the variable $t$, so that the path starts at $A$ when $t = 0$, ends at $B$ when $t  = 1$, and passes through $P_j$ at point $p_j = (p_{j,1}, p_{j,2}, ..., p_{j,k})$ when $t = t_j$, for all $1 \leq j \leq k$.  

If $q$ bends before $p_{i+1}$, then let $p_j$ be the first place that it bends.  By repeated applications of Lemma~\ref{lem:for_key_lem}, $q$ also passes through $P_i \cap P_{i+1}$ and we are done.
So assume that $q$ is a straight line from $A$ to $p_{i+1}$.  Thus, the $i$-th coordinate changes linearly from $a_i$ to $-b_i$, and from the parametrization of this, we get $t_{i+1} = \frac{a_i - p_{i+1,i}}{a_i + b_i}$.\\  
\\
Case 1: $p_{i+1,i+2} \neq 0 $ (That is, the shortest ordered path $q$ does not bend at $p_{i+1}$.)

In this case, $p_{i+1, i+1} = 0 = a_{i+1} + t_{i+1}(-b_{i+1} - a_{i+1}),$ which implies $t_{i+1} = \frac{a_{i+1}}{a_{i+1} + b_{i+1}}$.  Equate this value of $t_{i+1}$ with the one found above, and rearrange to get $p_{i+1,i} = a_i - \frac{a_{i+1}(a_i + b_i)}{a_{i+1} + b_{i+1}}$.  The definition of $P_{i+1}$ and the assumption $p_i \neq p_{i+1}$ implies that $p_{i+1,i} < 0$.  Hence, $a_i < \frac{a_{i+1}(a_i + b_i)}{a_{i+1} + b_{i+1}}$, which can be rearranged to $\frac{a_i}{b_i} < \frac{a_{i+1}}{b_{i+1}}$, a contradiction.\\
\\
Case 2: $p_{i+1,i+2} = 0$ (That is, the shortest ordered path $q$ bends at $p_{i+1}$, and $p_{i+1} = p_{i+2}$.)

Let $J \geq 2$ be the largest integer such that $p_{i+J} = p_{i+1}$, but $p_{i+J+1} \neq p_{i+1}$.   Apply Corollary~\ref{cor:homeo_to_Rm} using the partition $\{1\},\{2\}, ..., \{i\},\{i+1\}, \{i+2, ..., i+J\}, \{i+J+1\}, ..., \{k\}$ to reduce the space by $J-2$ dimensions.  $A$ and $B$ are mapped to $\widetilde{A} = (\widetilde{a}_1, ..., \widetilde{a}_{k-(J-2)})$ and $\widetilde{B} = (-\widetilde{b}_1, ...,-\widetilde{b}_{k-(J-2)})$, respectively, in the lower dimension space, where:
\begin{align*}
\widetilde{a}_j =
\begin{cases}
a_j & \text{if $j \leq i+1$} \\
\sqrt{\sum_{l = 2}^{J} a_{i+l}^2} & \text{if $j = i+2$} \\
a_{j+J - 2} & \text{if $j > i+2$}
\end{cases}
\quad \text{ and } \quad 
\widetilde{b}_j =
\begin{cases}
b_j & \text{if $j \leq i+1$} \\
\sqrt{\sum_{l = 2}^{J} b_{i+l}^2} & \text{if $j = i+2$} \\
b_{j+J - 2} & \text{if $j > i+2$}
\end{cases}
\end{align*}
Let $\widetilde{k} = k - (J-2)$.  Let $\widetilde{p}_j$ be the image of $p_j$ in $\R^{\widetilde{k}}$ under the above mapping if $j \leq i+2$ and the image of $p_{j+J-2}$ if $j > i+2$.  Let $\widetilde{P}_j =  \{ (x_1, ..., x_{\widetilde{k}}) \in \R^{\widetilde{k}}: x_l \leq 0 \text{ if $l < j$; } x_l = 0 \text{ if $l = j$; }  x_l \geq 0 \text{ if $l > j$} \}$.  So $\widetilde{P}_j$ is the boundary between the $j$-th and $(j+1)$-st orthants in the lower dimension space $\R^{\widetilde{k}}$.  Let $\widetilde{q}$ be the image of $q$.

Then $\widetilde{q}$ is a straight line from $\widetilde{A}$ to $\widetilde{p}_{i+1}$, and $\widetilde{p}_{i+1} = \widetilde{p}_{i+2} \neq \widetilde{p}_{i+3}$, so $\widetilde{q}$ bends in $\widetilde{P}_{i+1} \cap \widetilde{P}_{i+2}$.  Since $\widetilde{q}$ does not intersect $\widetilde{P}_{i+2} \cap \widetilde{P}_{i+3}$, by the contrapositive of Lemma~\ref{lem:for_key_lem}, $\frac{\widetilde{a}_{i+1}}{\widetilde{b}_{i+1}} > \frac{\widetilde{a}_{i+2}}{\widetilde{b}_{i+2}}$.  In $\R^k$, this translates into the condition that $\frac{a_{i+1}}{b_{i+1}} > \frac{\sqrt{\sum_{l = 2}^{J} a_{i+l}^2 }}{\sqrt{\sum_{l = 2}^{J} b_{i+l}^2}}$.   Cross-multiply, square each side, add $a_{i+1}^2 b_{i+1}^2$, and rearrange to get $\frac{a_{i+1}}{b_{i+1}} > \frac{\sqrt{\sum_{l = 1}^{J} a_{i+l}^2 }}{\sqrt{\sum_{l = 1}^{J} b_{i+l}^2}}$.

The remaining analysis is in $\R^k$.  If the shortest, ordered path is a straight line through $p_{i+1}$, then we make the same argument as in Case 1.  Otherwise, since the path does not bend at $p_i$, the $i$-th coordinate changes linearly from $a_i$ to $-b_i$.  We use this parametrization to find $t_{i+2} = t_{i+1} = \frac{a_i - p_{i+1,i}}{a_i + b_i}$.

Furthermore, the $(i+1)$-st to $(i+J)$-th coordinates decrease at the same rate from $A$ to $p_{i+1}$ and at the same, but possibly different than the first, rate from $p_{i+1}$ to $B$.  Therefore, we can apply Corollary~\ref{cor:homeo_to_Rm} to the partition $\{1\}, \{2\}, ..., \{i\}, \{i+1, i+2, ..., i+J\}, \{i+J+1\}, ..., \{k\} $ to isometrically map the shortest, ordered path into $\R^{m-(J-1)}$.  Let $\widetilde{a} = \sqrt{\sum_{l = 1}^{J} a_{i+l}^2}$, and let $\widetilde{b} = \sqrt{\sum_{l=1}^J(-b_{i+l})^2}$.  Then in $\R^{m - (J-1)}$, the $(i+1)$-st coordinate of the shortest ordered path changes at a constant rate from $\widetilde{a}$ to $- \widetilde{b}$.  This implies $0 = \widetilde{a} + t_{i+1}(-\widetilde{b} - \widetilde{a})$, or $t_{i+1} = \frac{\widetilde{a}}{\widetilde{a} + \widetilde{b}}$.  Equate the two expressions for $t_{i+1}$ to get $p_{i+1,i} = a_i - \frac{(a_i + b_i)\widetilde{a}}{\widetilde{a} + \widetilde{b}}$.  By definition of $P_{i+1}$, $p_{i+1,i} < 0$.  This implies $\frac{a_i}{b_i} < \frac{\widetilde{a}}{\widetilde{b}} = \frac{\sqrt{\sum_{l = 1}^{J} a_{i+l}^2}}{\sqrt{\sum_{l=1}^J(-b_{i+l})^2}}$.  But we showed that $\frac{\sqrt{\sum_{l = 1}^{J} a_{i+l}^2}}{\sqrt{\sum_{l=1}^J(-b_{i+l})^2}} < \frac{a_{i+1}}{b_{i+1}}$, so $\frac{a_i}{b_i} < \frac{a_{i+1}}{b_{i+1}}$, which is also a contradiction.
\end{proof}
By repeatedly applying this lemma, we find the lowest dimensional space containing the shortest, ordered path.  In this space, the ratios derived from the coordinates of the images of $A$ and $B$ form a non-descending sequence.  The following theorem gives the shortest path through $V(\R^k)$ from a point in the positive orthant to a point in the negative orthant, or equivalently, the shortest tour that passes through $P_1, ..., P_k$ in $\R^k$. 
\begin{thm}{} \label{th:get_dist}
Let $A = (a_1, a_2, ..., a_k)$ and $B = (-b_1, -b_2, ..., -b_k)$ with $a_i, b_i \geq 0$ for all $1 \leq i \leq k$ be points in $\R^k$.  Alternate between applying Lemma~\ref{lem:inductive_step_for_shortest_path_alg} and Corollary~\ref{cor:homeo_to_Rm} until there is a non-descending sequence of ratios $\frac{\widetilde{a}_1}{\widetilde{b}_1} \leq \frac{\widetilde{a}_2}{\widetilde{b}_2} \leq ... \leq \frac{\widetilde{a}_m}{\widetilde{b}_m}$, where $\widetilde{a}_i$ and $\widetilde{b}_i$ are the coordinates in the lower dimensional space.  There is a unique shortest path between $\widetilde{A} = (\widetilde{a}_1, ...., \widetilde{a}_m)$ and $\widetilde{B} = (-\widetilde{b}_1, ..., -\widetilde{b}_m)$ in $V(\R^m)$, with distance $\sqrt{\sum_{i=1}^m(\widetilde{a}_i + \widetilde{b}_i)^2}$.  This is the length of the shortest path between $A$ and $B$ in $V(\R^k)$.
\end{thm}
\begin{proof}
For the smallest $i$ such that $\frac{a_i}{b_i} > \frac{a_{i+1}}{b_{i+1}}$, Lemma~\ref{lem:inductive_step_for_shortest_path_alg} implies that $p_i = p_{i+1}$ in the shortest, ordered path in $\R^k$.  Thus, we can isometrically map this problem to the space one dimension lower that results from applying Corollary~\ref{cor:homeo_to_Rm} using the partition $\{1\}, \{2\}, ..., \{i-1\},\{i,i+1\}, \{i+2\}, ..., \{m\}$.  We repeat these two steps, iteratively mapping this problem to lower dimensional spaces, until the new ratio sequence is non-descending.  Let $\frac{\widetilde{a}_1}{\widetilde{b}_1} \leq \frac{\widetilde{a}_2}{\widetilde{b}_2} \leq ... \leq \frac{\widetilde{a}_m}{\widetilde{b}_m}$ be this ratio sequence.  By Lemma~\ref{lem:ascending_ratio_seq_means_geo_is_line}, the geodesic between $\widetilde{A}$ and $\widetilde{B}$ is the straight line.  Furthermore, its length is  $\sqrt{\sum_{i=1}^m(\widetilde{a}_i + \widetilde{b}_i)^2}$.  Since we mapped from $V(R^k)$ to $V(R^m)$ by repeated isometries, both the length of the path and the order it passes through $P_1, ..., P_m$, or their images, remain the same.  Thus the pre-image of this path is the shortest path in $V(R^k)$.
\end{proof}
%
%
%
\subsubsection{\linalg{}:  A Linear Algorithm for Computing Path Space Geodesics}
\label{sec:linalg}
Theorem~\ref{th:get_dist} can be translated into a linear algorithm called \linalg{}, for computing the path space geodesic between $T_1$ and $T_2$ through some path space $S = \cup_{i=0}^k \Or_k$.  For all $1 \leq i \leq k$, let $A_i = E_{i-1} \backslash E_i$ and $B_i = F_i \backslash F_{i-1}$, and let $a_i = \norm{A_i}$ and $b_i = \norm{B_i}$.

Let $1 \leq i < k$ be the least integer such that $\frac{a_i}{b_i} > \frac{a_{i+1}}{b_{i+1}}$.  Then by Theorem~\ref{th:get_dist}, to find the path space geodesic through $S$, we should apply Lemma~\ref{lem:inductive_step_for_shortest_path_alg} and Corollary~\ref{cor:homeo_to_Rm} to the ratio sequence $ \frac{a_1}{b_1}, \frac{a_2}{b_2}, ..., \frac{a_k}{b_k} $ to map the problem to $V(\R^{k-1})$, where the ratio sequence becomes $\frac{a_1}{b_1}, ..., \frac{a_{i-1}}{b_{i-1}}, \frac{ \sqrt{a_i^2 + a_{i+1}^2}} { \sqrt{b_i^2 + b_{i+1}^2}}, \frac{a_{i+2}}{b_{i+2}}  ..., \frac{a_k}{b_k} $.  Repeat this process until the ratio sequence is non-descending.

Unfortunately, this process is not deterministic, in that different non-descending ratio sequences can be found for the same geodesic, depending on the starting path space.  This occurs, because by Corollary~\ref{cor:equal_ratios}, two equal ratios can be combined to give a ratio sequence corresponding to a path with the same length.  However, if we modify the algorithm to also combine equal ratios, the output ascending ratio sequence will be unique for a given geodesic.

Define the \emph{carrier of the path space geodesic through $S$} between $T_1$ and $T_2$ to be the path space $Q = \cup_{i=0}^l \Or_{c(i)} \subseteq S$ such that the path space geodesic through $S$ traverses the relative interiors of $\Or_{c(0)}$, $\Or_{c(1)}$, ..., $\Or_{c(l)}$, where the function $c:\{0, 1, ..., l \} \to \{0, ...,k\}$ takes $i$ to $c(i)$ if the $i$-th orthant is $Q$ is the $c(i)$-th orthant in $S$.  If a path space geodesic is the geodesic, we just write \emph{carrier of the geodesic}.  Then the carrier of the path space geodesic is the path space whose corresponding ratio sequence is the unique ascending ratio sequence for the path space geodesic.

We now explicitly describe the algorithm for computing the ascending ratio sequence corresponding to the path space geodesic, \linalg{}, and prove it has linear runtime.
\\
\\
\linalg{} \\
\emph{Input}:  Path space $S$ or its corresponding ratio sequence $R = \frac{a_1}{b_1}, \frac{a_2}{b_2}, ..., \frac{a_k}{b_k} $  \\
\emph{Output}:  The path space geodesic, represented as an ascending ratio sequence, which is understood to be the partition of $R$ where the ratio $\frac{\sqrt{\sum_{j = 0}^J a_{i+j}^2}}{\sqrt{\sum_{j = 0}^J b_{i+j}^2}}$ corresponds to the block $\left \{  \frac{a_i}{b_i}, \frac{a_{i+1}}{b_{i+1}}, ..., \frac{a_{i+J}}{b_{i+J}} \right \}$.\\
\emph{Algorithm}:  
Starting with the ratio pair $\frac{a_1}{b_1}, \frac{a_2}{b_2}$, \linalg{} compares consecutive ratios.  If for the $i$-th pair, we have $\frac{a_i}{b_i} \geq \frac{a_{i+1}}{b_{i+1}}$, then \emph{combine} the two ratios by replacing them by $\frac{ \sqrt{a_i^2 + a_{i+1}^2} }{ \sqrt{ b_i^2 + b_{i+1}^2 } }$ in the ratio sequence.  Compare this new, combined ratio with the previous ratio in the sequence, and combine these two ratios if they are not ascending.  Again the newly combined ratio must be compared with the ratio before it in the sequence, and so on.  Once the last combined ratio is strictly greater then the previous one in the sequence, we again start moving forward through the ratio sequence, comparing consecutive ratios.  The algorithm ends when it reaches the end of the ratio sequence, and the ratios form an ascending ratio sequence.
\begin{thm}\label{th:PATHGEOcomplexity}
\linalg{} has complexity $\Theta(k)$, where $k+1$ is the number of orthants in the path space between $T_1$ and $T_2$.
\end{thm}
\begin{proof}
We first show the complexity is $O(k)$.  Combining two ratios reduces the number of ratios by 1, so this operation is done at most $k-1 = O(k)$ times.  It remains to count the number of comparisons between ratios.  Each ratio is involved in a comparison when it is first encountered in the sequence.  There are $k-1$ such comparisons.  All other comparisons occur after ratios are combined, so there are at most $k-1$ of these comparisons.  Therefore, \linalg{} has complexity $O(k)$.  Any algorithm must make $k-1$ comparisons to ensure the ratios are in ascending order, so the complexity is $\Omega(k)$, and thus this bound is tight.
\end{proof}
%
%
%
%
%
%
\section{Algorithms}
In this section, we show in Theorem~\ref{th:can_use_dp} how to compute the geodesic distance between two trees $T_1$ and $T_2$ by computing the geodesic between certain smaller, related trees.  This allows us to use the results from Sections 3 and 4, as well as either dynamic programming or divide and conquer techniques, to devise two algorithms for finding the geodesic between two trees with no common splits.  Experiments on random trees show these algorithms are exponential, but practical on trees with up to 40 leaves, as well as larger trees from biological data.

%
%
\subsection{A Relation between Geodesics}
Let $T_1$ and $T_2$ be two trees in $\T_n$ with no common splits.  The following theorem shows that there exists a path space containing the geodesic between $T_1$ and $T_2$ such that a certain subspace of it contains the geodesic between two smaller, related trees, $T_1'$ and $T_2'$.  As $T_1'$ and $T_2'$ have fewer splits than $T_1$ and $T_2$, it is easier to compute this geodesic.  Therefore, we can find the geodesic between $T_1$ and $T_2$ by finding the geodesic between all such possible $T_1'$ and $T_2'$.

\begin{defn}
Let $S = \cup_{i=0}^k \Or_i$ be a path space between $T_1$ and $T_2$.  Define $r(S) = \cup_{i=0}^{k-1} \Or(E_i \bs E_{k-1} \cup F_i)$ to be the \emph{truncation} of $S$.
\end{defn}

Then $r(S)$ is a path space between $T_1' = T(\Sigma_1 \bs E_{k-1})$ and $T_2' = T(F_{k-1})$.  That is, $T_1'$ and $T_2' $ are exactly the trees $T_1$ and $T_2$ with the edges $E_{k-1} = A_k$ and $F_k \bs F_{k-1} = B_k$ contracted, and $r(S)$ is the subspace of $S$ formed by removing all trees having edges in $A_k$ or $B_k$ of non-zero length.  Finally, if the path space $S' = \cup_{i=0}^{k-1} \Or_i'$ is the truncation of a path space between trees $T_1$ and $T_2$, then there is a unique path space $S = \cup_{i = 0}^{k-1}  \left ( \Or_i' + \Or(\Sigma_1 \bs E_{k-1}') \right ) \cup \Or(\Sigma_2)$ between $T_1$ and $T_2$ such that $r(S) = S'$.

\begin{thm}
\label{th:can_use_dp}

Let $T_1$ and $T_2$ be two trees in $\T_n$ with no common splits.  Then there exists a path space $Q = \cup_{i=0}^k \Or_i$ that contains the geodesic between $T_1$ and $T_2$, such that the truncation $Q' = r(Q)$ is the carrier of the geodesic between $T_1' = T(\Sigma_1 \bs E_{k-1})$ and $T_2' = T(F_{k-1})$.
\end{thm}
To prove this theorem, we first prove two lemmas which hold for any path space $S$ between $T_1$ and $T_2$, with truncation $S'$.   
Lemma~\ref{lem:proj_QM} shows that the path space geodesic through $S$ is contained in a path space whose truncation is the carrier of the path space geodesic of $S'$.  Lemma~\ref{lem:can_use_dp} shows that if $S'$ does not contain the geodesic between $T_1'$ and $T_2'$, and hence we can find another path space $P'$ containing a shorter path space geodesic, then the corresponding path space $P$ between $T_1$ and $T_2$ does not contain a path space geodesic longer than the one in $S$.
\begin{lem}
\label{lem:proj_QM}
Let $T_1$ and $T_2$ be two trees in $\T_n$ with no common splits, and let $S = \cup_{i = 0}^k \Or_i$ be a path space between them.
Let $Q '$ be the carrier of the path space geodesic through $S' = r(S)$ between $T_1' = T(\Sigma_1 \bs E_{k-1})$ and $T_2' = T(F_{k-1})$.  Let $Q$ be the path space between $T_1$ and $T_2$ such that $r(Q) = Q'$.  Then $d_Q(T_1, T_2) = d_S(T_1, T_2)$.
\end{lem}
\begin{proof}
Since $Q'$ is the carrier of the path space geodesic through $S'$, both $Q'$ and $S'$ have the same path space geodesic, and hence \text{\linalg} will return the same ascending ratio sequence for either input $Q'$ or $S'$.  Let this ascending ratio sequence be $\frac{a_1'}{b_1'}, \frac{a_2'}{b_2'}, \ldots, \frac{a_l'}{b_l'}$.  The ratio sequences corresponding to the path spaces $Q$ and $S$ are just the ratio sequences for $Q'$ and $S'$, respectively, with the ratio $\frac{ \norm{A_k}}{ \norm{B_k}}$ added to the end of each.  So for both inputs $Q$ and $S$, the ratio sequence when \linalg{} compares $\frac{ \norm{A_k}}{ \norm{B_k}}$ for the first time is $\frac{a_1'}{b_1'}, \frac{a_2'}{b_2'}, \ldots, \frac{a_l'}{b_l'}, \frac{ \norm{A_k}}{ \norm{B_k}}$.  This implies that the ratio sequence output by $\text{\linalg}(Q)$ is the same as that output by $\text{\linalg}(S)$, and hence $d_Q(T_1, T_2) = d_S(T_1, T_2)$.
\end{proof}
\begin{lem}
\label{lem:can_use_dp}
Let $T_1$ and $T_2$ be two trees in $\T_n$ with no common splits, and let $S$ be a path space between them.
If $S' = r(S)$ does not contain the geodesic between $T_1'= T(\Sigma_1 \bs E_{k-1})$ and $T_2' = T(F_{k-1})$, then there exists a path space $P'$ between $T_1'$ and $T_2'$  such that $d_{P'}(T_1',T_2') < d_{S'}(T_1',T_2')$ and $d_{P}(T_1, T_2) \leq d_S(T_1, T_2)$, where $P$ is the path space between $T_1$ and $T_2$ with truncation $P'$.  
\end{lem}
\begin{proof}
Let $S' = \cup_{i= 0}^l \Or_i'$, and let $Q '= \cup_{i=0}^l\Or_{c(i)}'$ be the carrier of the path space geodesic through $S'$.  Let $q$ be the path space geodesic through $Q'$ between $T_1'$ and $T_2'$, and let $q_i = \Or_{c(i-1)}' \cap \Or_{c(i)}' \cap q$ for every $1 \leq i \leq l$.  Since $q$ is not the geodesic from $T_1'$ to $T_2'$, $q$ cannot be locally shortest in $\T_n$.  By Proposition~\ref{prop:local_geo_1}, for all $1 \leq i \leq l-1$, the part of $q$ between $q_i$ and $q_{i+1}$ is a line, and cannot be made shorter in $\T_n$.  Thus we can only find a locally shorter path in $\T_n$ by varying $q$ in the neighbourhood of some $q_j$.  In particular, there exists some $\varepsilon$ such that if $s$ and $t$ are the points on $q$, $\varepsilon$ before and after $q_j$ in the orthants $\Or_{c(j-1)}$ and $\Or_{c(j)}$, respectively, then the geodesic between $s$ and $t$ does not follow $q$.  Replace the part of $q$ between $s$ and $t$ with the true geodesic between $s$ and $t$ to get a shorter path in $\T_n$, with distance $d_s$.  Let $\Or_{c(j-1)}, \Or_1'' = \Or(E_1'' \cup F_1''), ..., \Or_m'' = \Or(E_m'' \cup F_m''), \Or_{c(j)}$ be the sequence of orthants through whose relative interiors the geodesic between $s$ and $t$ passes.  Note that $\Or_1'', ..., \Or_m''$ are not in $S'$.  These orthants must form a path space, and thus $P' = Q' \cup \left ( \cup_{i=0}^m \Or_i'' \right )$ is a path space.  Since the path space geodesic is the shortest path through a path space, $d_{P'}(T_1', T_2') \leq d_s < d_{Q'}(T_1', T_2')$.   By definition of $Q'$, $d_{Q'}(T_1', T_2') = d_{S'}(T_1', T_2')$, and hence $d_{P'}(T_1', T_2') < d_{S'}(T_1', T_2')$, as desired.

To show that $d_{P}(T_1, T_2) \leq d_S(T_1, T_2)$, let $Q$ be the path space between $T_1$ and $T_2$ such that $r(Q) = Q'$.  Then $Q \subset P$, which implies $d_{P}(T_1, T_2) \leq d_{Q}(T_1, T_2)$.  By Lemma~\ref{lem:proj_QM}, $d_{Q}(T_1, T_2) = d_S(T_1, T_2)$, and so $P'$ is the desired path space.
\end{proof}

We use Lemma~\ref{lem:proj_QM} and Lemma~\ref{lem:can_use_dp} to prove Theorem~\ref{th:can_use_dp}.
\begin{proof}[Proof of Theorem~\ref{th:can_use_dp}]
We first show there exists a path space $M$ containing the geodesic between $T_1$ and $T_2$, such that its truncation $M'$ contains the geodesic between $T_1'$ and $T_2'$.  So let $S$ be any path space containing the geodesic between $T_1$ and $T_2$, with truncation $S' = r(S)$.  If $S'$ contains the geodesic between $T_1'$ and $T_2'$, then we are done.  If not, then by Lemma~\ref{lem:can_use_dp}, there exists a path space $P'$ from $T_1'$ to $T_2'$ with $d_{P'}(T_1', T_2') < d_{S'}(T_1', T_2')$ and $d_P(T_1, T_2) \leq d_S(T_1, T_2)$, where $P$ is the path space between $T_1$ and $T_2$ such that $P' = r(P)$.  Since $S$ contains the geodesic from $T_1$ to $T_2$, we have $d_P(T_1, T_2) = d_S(T_1, T_2)$, and hence $P$ also contains the geodesic.  If $P'$ contains the geodesic between $T_1'$ and $T_2'$, then we are done.  Otherwise,  repeat this step by applying Lemma~\ref{lem:can_use_dp} to $P$ and $P'$.  This process produces a path space containing a strictly shorter path space geodesic at each iteration, so since there are only a finite number of path spaces, it eventually finds a path space containing the geodesic from $T_1'$ to $T_2'$.  

Let $Q'$ be the carrier of the path space $M'$ containing the geodesic between $T_1'$ and $T_2'$, and let $Q$ be the path space from $T_1$ to $T_2$ such that $r(Q) = Q'$.  Then by Lemma~\ref{lem:proj_QM}, $Q$ also contains the geodesic between $T_1$ and $T_2$, and we are done. 
\end{proof}
 We will now present two algorithms for computing geodesics.  Both of these algorithms use Theorem~\ref{th:can_use_dp} to avoid computing the path space geodesic for every maximal path space between $T_1$ and $T_2$.  This significantly decreases the runtime.  We call these algorithms \textsc{GeodeMaps}, which stands for GEOdesic DistancE via MAximal Path Spaces.  The first algorithm uses dynamic programming techniques, and is denoted \DPalg, while the second uses a divide and conquer strategy, and is denoted \DCalg.
 
\subsection{\DPalg:  a Dynamic Programming Algorithm} 
Theorem~\ref{th:can_use_dp} implies that we can find the geodesic between trees $T_1$ and $T_2$ by just considering certain geodesics corresponding to the elements covered by $\Sigma_2$ in $K(\Sigma_1, \Sigma_2)$.  More specifically, for any $A \in K(\Sigma_1, \Sigma_2)$ covered by $\Sigma_2$, let $Q_A'$ be the carrier for the geodesic $g_A$ from $T(X_{\Sigma_1}(A))$ to $T(A)$.  Then the geodesic from $T_1$ to $T_2$ is the minimum-length path space geodesic through the path spaces $\{Q_A: A \in K(\Sigma_1, \Sigma_2) \text{ is covered by $\Sigma_2$ and } Q_A' = r(Q_A) \}$. 

An analogous method can be applied to find the geodesic $g_A$.  In general, for any element $A \neq \varnothing$ in $K(\Sigma_1, \Sigma_2)$, the geodesic between trees $T(X_{\Sigma_1}(A))$ and $T(A)$ can be computed from the carriers $Q_B'$ of the geodesics from $T(X_{\Sigma_1}(B))$  to $T(B)$ for each $B$ covered by $A$.  This is done by finding the minimum-length path space geodesic through the path spaces $\{Q_B: B \in K(\Sigma_1, \Sigma_2) \text{ is covered by $A$ and } Q_B' = r(Q_B) \}$.  

This suggests the following algorithm.  Let $G_{K(\Sigma_1, \Sigma_2)}$ be the directed graph with vertices in bijection with the elements of $K(\Sigma_1, \Sigma_2)$, and with an edge between two vertices if and only if there is a cover relation between their corresponding elements in $K(\Sigma_1, \Sigma_2)$.  The edge is directed from the covered element to the covering element.  Then we can compute the geodesic distance by doing a breath-first search on $G_{K(\Sigma_1, \Sigma_2)}$.  As we visit each node $A$ in $G_{K(\Sigma_1, \Sigma_2)}$, we construct the geodesic between $T(X_{\Sigma_1}(A))$ and $T(A)$ using the geodesics between $T(X_{\Sigma_1}(B))$ and $T(B)$ for each $B$ covered by $A$.  This algorithm visits every node in the graph, of which there can be an exponential number as shown in Remark~\ref{rem:exponential_path_poset}, so this algorithm is exponential in the worst case.  However, this is a significant improvement over considering each maximal path space.

We implemented a more memory-efficient version of this algorithm, called \DPalg.  This version uses a depth-first search of $G_{K(\Sigma_1,\Sigma_2)}$.  For each element $A$ in $G_{K(\Sigma_1,\Sigma_2)}$, store the distance of the shortest path space geodesic found so far between $T(X_{\Sigma_1}(A))$ and $T(A)$.  If \DPalg{} revisits an element with a longer path space geodesic, it prunes this branch of the search. 

\DPalg{} stores the carrier of the shortest path space geodesic found so far between $T_1$ and $T_2$.  As a heuristic improvement, at each step in the depth-first search, \DPalg{} chooses the node with the lowest transition ratio of the nodes not yet visited.  For more details and an example of \DPalg{}, see \cite[Section 5.2.1]{thesis}.


 \subsection{\DCalg:  a Divide And Conquer Algorithm}
If $A$ is an element in $K(\Sigma_1, \Sigma_2)$, then the trees in the corresponding orthant share the splits $A$ with $T_2$.  This inspires the following algorithm, which we call \DCalg.  Choose some minimal element of $P(\Sigma_1, \Sigma_2)$, and add the splits in this equivalence class to $T_1$ by first dropping the incompatible splits.  For example, if we choose to add the split set $F_1$, then we must drop $X_{\Sigma_1}(F_1)$.  The trees with this new topology now have splits $F_1$ in common with $T_2$.  Apply Theorem~\ref{th:dist_if_common_edge} to divide the problem into subproblems along these common splits.  For each subproblem, recursively call \DCalg.  Since some subproblems will be encountered many times, store the geodesics for each solved subproblem in a hash table.
 
Each subproblem corresponds to an element in $K(\Sigma_1, \Sigma_2)$, and \DCalg{} is polynomial in the number of subproblems solved.  Hence an upper bound on the complexity of \DCalg{} is the number of elements in $K(\Sigma_1, \Sigma_2)$, which is exponential in general by Remark~\ref{rem:exponential_path_poset}.  See \cite[Section 5.2.2]{thesis} for details of this algorithm, an example, and a family of trees for which \DPalg{} has exponential runtime.

\subsection{Performance of \DPalg{} and \DCalg{}}
We now compare the runtime performance of \DPalg{} and \DCalg{} with \textsc{GeoMeTree} \cite{KHK08}, the only other geodesic distance algorithm published when this paper was written.  For $n = 10, 15, 20, 25, 30, 35, 40, 45$, we generated 200 random rooted trees with $n$ leaves, using a birth-death process.  Specifically, we ran \textsc{evolver}, part of PAML \cite{Yang07} with the parameters estimated for the phylogeny of primates in \cite{YangRannala97}, that is 6.7 for the birth rate ($\lambda$), 2.5 for the death rate ($\mu$), 0.3333 for the sampling rate, and 0.24 for the mutation rate.  For each $n$, we divided the 200 trees into 100 pairs, and computed the geodesic distance between each pair.  The average computation times are given in Figure~\ref{fig:graphs}.  Memory was the limiting factor for all three algorithms, and prevented us from calculating the missing data points.
\begin{figure}[ht]
\vspace{-1cm}
\centering
\includegraphics[scale = 0.4]{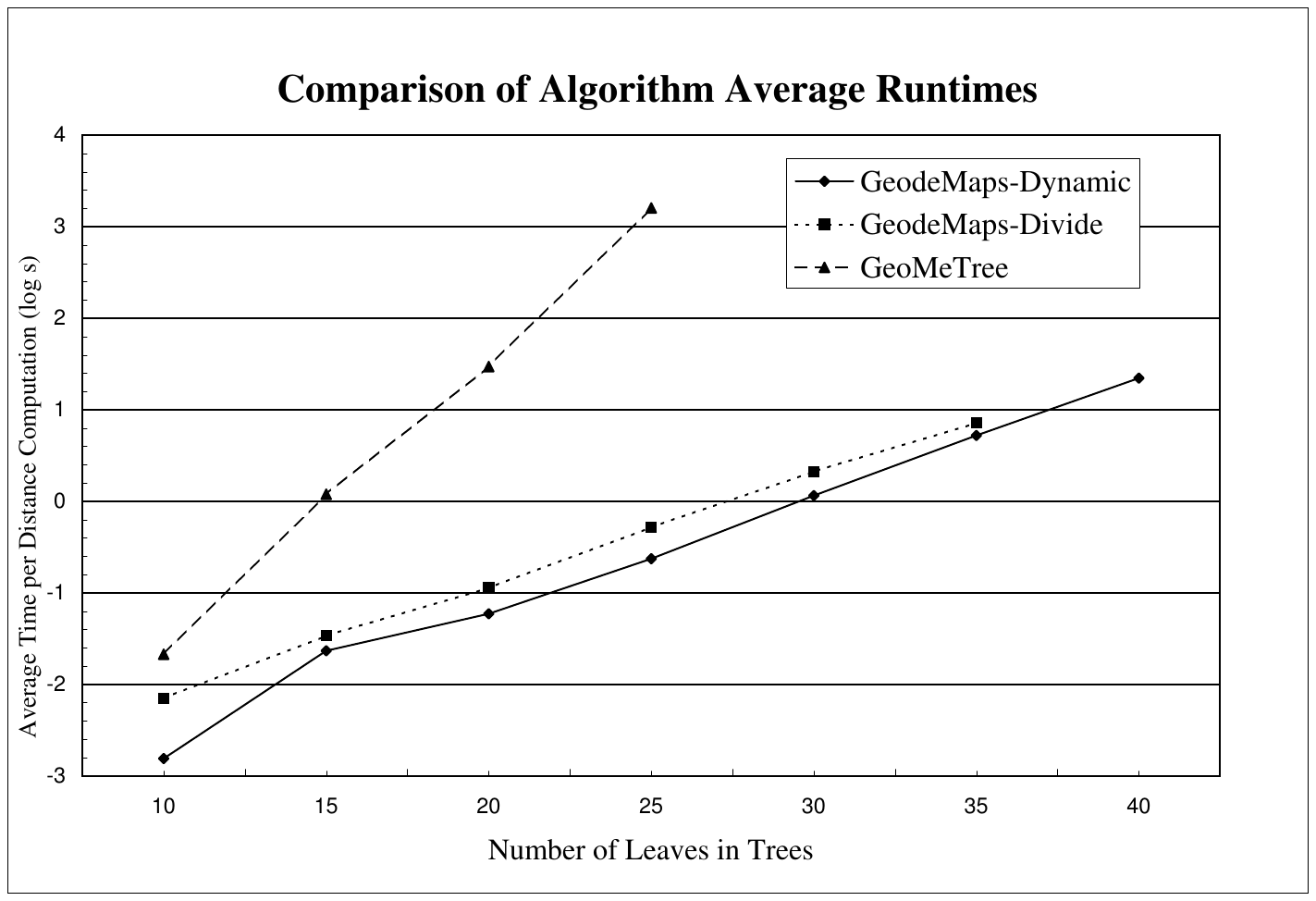}
\vspace{-1cm}
\caption{Average runtimes of the three geodesic distance algorithms.}
\label{fig:graphs}
\end{figure}
Both \DPalg{} and \DCalg{} exhibit exponential runtime, but they are significantly faster the \textsc{GeoMeTree}.  Note that as the trees used were random, they have very few common splits.  Biologically meaningful trees often have many common splits, resulting in much faster runtimes.  For example, for a data set of 31 43-leaved trees representing possible ancestral histories of bacteria and archaea \cite{L07}, we computed the geodesic distance between each pair of trees.  Using \DPalg{} the average computation time was 0.531 s, while using \DCalg{} the average time was 0.23 s.  This contrasts to an average computation time of 22 s by \DPalg{} for two random trees with 40 leaves.
All computations were done on a Dell PowerEdge Quadcore with 4.0 GB memory, and 2.66 GHz x 4 processing speed.  The implementation of these algorithms, GeodeMaps 0.2, is available for download from www.math.berkeley.edu/\~{}megan/geodemaps.html.

%
%
\section{Conclusion}
We have used the combinatorics and geometry of the tree space $\T_n$ to develop two algorithms to compute the geodesic distance between two trees in this space.  In doing so, we developed a poset representation for the possible orthant sequences containing the geodesic, and gave a linear time algorithm for computing the shortest path in the subspace $V(\R^n)$ of $\R^n$, which will help characterize when the general problem of finding the shortest path through $\R^n$ with obstacles is NP-hard.  We also showed that geodesics can be computed by solving smaller subproblems.


\subsection*{Acknowledgements}
We thank Louis Billera for numerous helpful discussions and suggestions about this work; Karen Vogtmann for sharing her notes and thoughts on the problem; Seth Sullivant for suggestions that greatly improved the presentation of this work; Philippe Lopez for the kind provision of the biological data set; Joe Mitchell for pointing out that finding the geodesic in $V(\R^k)$ is equivalent to solving a touring problem; and an anonymous referee for constructive and helpful comments.
\bibliography{thesis_paper_bib}
\end{document}